 \documentclass[draft]{article}

\usepackage{amsmath,amsfonts,amsthm,amssymb,amscd,cancel,color}
\usepackage{enumitem}
\usepackage{verbatim}
\usepackage{ulem,color}

\renewcommand{\Re}{\mathop{\rm Re}\nolimits}
\renewcommand{\Im}{\mathop{\rm Im}\nolimits}

\theoremstyle{plain} \newtheorem{theorem}{Theorem}[section]
\newtheorem{lemma}[theorem]{Lemma}
\newtheorem{proposition}[theorem]{Proposition}
\newtheorem{corollary}[theorem]{Corollary} \theoremstyle{definition}
\newtheorem{definition}[theorem]{Definition} \theoremstyle{remark}
\newtheorem{remark}[theorem]{Remark}

\newtheorem{claim}[theorem]{Claim}

\newcommand{\R}{{\mathbb R}}

\newcommand{\Z}{{\mathbb Z}}

\newcommand{\N}{{\mathbb N}}

\def\im{{\rm i}}

\newcommand{\C}{\mathbb{C}}

\def\({\left(}
\def\){\right)}
\def\<{\left\langle}
\def\>{\right\rangle}
\newcommand{\cN}{\mathcal{N}}

\newcommand{\wto}{\rightharpoonup}

\newcommand{\cD}{\mathcal{D}}
\newcommand{\stz}{\mathrm{Stz}}
\newcommand{\st}{\mathfrak{st}}
\newcommand{\stzo}{\mathrm{Stz}^1}

\newcommand{\stzh}{\mathrm{Stz}^{1/2}}
\newcommand{\stzt}{\mathrm{Stz}^{\theta}}

\newcommand{\Ht}{H^{\theta}}
\newcommand{\Hh}{H^{1/2}}

\newcommand{\cNh}{\mathcal N_{1/2}}
\newcommand{\cNt}{\mathcal N_{\theta}}

\newcommand{\sjp}{s^j_{+,n}}
\newcommand{\sjm}{s^j_{-,n}}

\newcommand{\sjpm}{s^{j+1}_{-,n}}

\newcommand{\Lj}{\Lambda^{j}_n}
\newcommand{\lj}{\lambda^{j}_n}

\newcommand{\nnorm}[2]{[#1,#2]}
\numberwithin{equation}{section}

\begin{document}

\title{On nonlinear profile decompositions and scattering  for  a  NLS-ODE model}

\author {Scipio Cuccagna, Masaya Maeda}

\maketitle

\begin{abstract}
In this paper, we consider a Hamiltonian system combining a nonlinear Schr\" odinger equation (NLS) and  an ordinary differential equation (ODE).
This system is a simplified model of the NLS around soliton solutions.
Following Nakanishi \cite{NakanishiJMSJ}, we show scattering of $L^2$ small $H^1$ radial solutions.
The proof is based on Nakanishi's framework and Fermi Golden Rule estimates on $L^4$ in time norms.
\end{abstract}

\section{Introduction}

The analysis  of the asymptotic behavior for $t\to +\infty$ of solutions of
nonlinear dispersive equations is largely an open problem.  The  \textit{Soliton Resolution Conjecture} (SR Conjecture)    states    that    generic solutions of nonlinear dispersive equations in Euclidean spaces in the long time limit resolve  into  trains of solitons plus a dispersing radiative component.  For a review we refer to \cite{soffer}.  While the  conjecture itself is unsolved,  there is a large literature   studying scattering (possibly modulo solitons) for some specific equations and in some subsets  of phase space   invariant for the dynamics.
We emphasize two   lines of   research.

The first, starting from Buslaev-Perelman \cite{busper2}   and Soffer-Weinstein \cite{SW99IM}, considers invariant sets which are rather small and devotes attention to the so called   meta-stable torii.
They   vanish after a long time and their  anomalously weak  instability is governed by purely nonlinear interactions, a phenomenon    often called \textit{Radiation Damping}. %Notice that  apart from Nakanishi  \cite{NakanishiJMSJ} all papers following this line of research exploit a purely perturbation theory approach because they work in region of phase space close to a soliton.
The linearized   and   nonlinear dynamics  are completely different,  because   meta--stable torii do  not vanish in the linearized equation.

The second line of   research, starting from Kenig-Merle \cite{KM06Invent}, centers around the so-called \textit{Concentration Compactness Rigidity method} (CCR method) and aims to study large regions of   phase space.
The  main idea  is that a   solution  splits   into components  well separated from each other. In  \cite{KM06Invent,holmer} the method is used to prove the
scattering of solutions with   norm   smaller  than some critical and not small value.
  In \cite{merle1, merle2, merle3}, devoted to arbitrarily large solutions of energy critical equations,
the  components are either scattering or  are solitons.
The  proofs could be conditional   on the absence of discrete \textit{internal modes} (in the terminology of \cite{dmitry}) in whose presence typical tools like the so called   virial inequalities    have not been developed yet when meta--stable torii arise,    except in \cite{munoz}, a paper  which considers only
a single  discrete coordinate  rather simple in terms of the  combinatorial  structure of the normal form argument needed in the proof of Radiation Damping.
Related to these considerations is the fact that the SR Conjecture is known to fail  for systems such as discrete NLS's \cite{MA94N, JA97N, Maeda17SIMA} exactly because of the way Radiation Damping occurs or fails to occur. So one could envisage that between integrable systems, where the SR Conjecture is essentially known to be true and   no  {internal modes} are expected to exist,  \cite{Kevrekidis}, on one end   and some discrete equations on the other end there might be intermediate cases, which might be typical,   where the     SR Conjecture is correct but   requires
  an explicit elucidation of   Radiation Damping.

The CCR method has been applied also to other settings and, for instance, for wave maps   we refer to \cite{kriegerschlag}.

 Nakanishi's recent paper \cite{NakanishiJMSJ}  puts together the two distinct lines of research   described above   for a problem which, while featuring an eigenvalue which complicates the  CCR method, nonetheless  does not  have meta--stable torii. Our aim is   to   initiate a theory of the   CCR method for equations
which have meta--stable torii.
Specifically in this paper  we consider the following   NLS-ODE model:
\begin{align}
\im \dot \xi  &= -\Delta \xi + |\xi|^2\xi + |z|^2z   G, \label{1}\\
\im \dot z & = z + \frac 1 2 z^2 ( G|\xi) +|z|^2\overline{( G|\xi)},\label{2}
\end{align}
where $\xi(t)\in H^1(\R^3;\C)$, $z(t)\in \C$,  $(f|g):=\int_{\R^3}f\bar g\,dx$ and $G(x)\in  \mathcal{S}(\R ^3,\C )$ (Schwartz function) is a given radially symmetric function.

Schr\"odinger equations coupled with ODEs naturally appear  in the study of asymptotic stability of solitons of NLS   (see, for example \cite{SW1, SW2} or \cite{CM16JNS} and therein for more recent references).
The forcing term $|z|^2z G$, which governs the interaction between the PDE part and the ODE part, creates the radiation damping.
Moreover, such kind of model (with different interaction terms) appears in the study of particle-field interaction \cite{KK06RJMP} and models of friction \cite{EZ12JMP, FG14AM, FGS11JMP, FGS12CMP}.
There are also studies with the Schr\"odinger equation replaced by wave, Klein-Gordon and Dirac equations \cite{KSK97CPDE, IKM03JMP, IKV06CMP, KKS11JMAA}.

The system \eqref{1}-\eqref{2} is   Hamiltonian   with symplectic form
\begin{align}\label{3}
\Omega=\<\im d\xi, d\xi\> + \im dz\wedge d\bar z
\end{align}
and Hamiltonian function, for $\<f,g\>=\Re(f|g) $, given by
\begin{align}\label{4}
\mathbb{E}(\xi,z)=\frac 1 2 \|\nabla \xi \|_{L^2}^2 + \frac 1 4 \|\xi\|_{L^4}^4 + |z|^2 + \<|z|^2z  G,\xi\> .
\end{align}
Notice that for any $\vartheta \in \R $ the symplectic form $\Omega$ and the energy $\mathbb{E}$ are invariant
with respect to the diffeomorphism $(\xi,z )\to  (e ^{\im \vartheta}\xi ,e ^{\im \vartheta}z)$. Then the following quadratic
form is an invariant of motion for the system \eqref{1}-\eqref{2}:
\begin{equation}\label{eq:mass}
 \mathbb{M}(\xi,z)=-\frac 1 2  \left .\frac{d}{d\vartheta} \Omega  (  (e ^{\im \vartheta}\xi,e ^{\im \vartheta}z) , ( \xi, z)) \right | _{\vartheta =0}= \frac 1 2 \|  \xi \|_{L^2}^2 +     |z|^2.
\end{equation}
By standard arguments, see \cite{caz}, and using the conservation of $\mathbb E$ and $\mathbb M$ it is easy to conclude that the Cauchy problem is globally well posed in  $ H^1 (\R^3;\C)\times \C$ for the system \eqref{1}-\eqref{2}. The subspace $ H^1 _{rad}(\R^3;\C )\times \C$  is invariant for the flow.

\noindent We will assume the following,  true for most $G\in \mathcal{S}(\R ^3,\C )$:
\begin{equation}\label{eq:FGR}
 \left . \widehat{G}  \right | _{\{ y\in\R^3: |y |=1\}}\not \equiv 0,\text{  where  }\hat G(y):=\int_{\R^3}e^{-\im xy}G(x)\,dx.
\end{equation}
As mentioned before, our aim is to show scattering in a large region of phase space.
The assumption \eqref{eq:FGR}, ensures   that the system \eqref{1}-\eqref{2} exhibits radiation damping.
That is, even though the system is time reversible and Hamiltonian, there is a flow of mass from the ODE part to the Schr\"odinger part and $|z(t)|$ converges to $0$ as $t\to \pm \infty$.
%Further, the mass in the PDE part is damped into the spatial infinity.
%Such phenomenon was first investigated in \cite{SW99IM} and now considered to be responsible to several dynamical behavior of nonlinear dispersive equations such as asymptotic stability of solitons.

We will now introduce the precise definition of scattering.
  \begin{definition}[Scattering]\label{def:1}
Let $(\xi,z)$ be solution of system \eqref{1}--\eqref{2}.
We say $(\xi,z)$ scatters forward (resp.\ backward) in time if there exists $\varphi \in H^1(\R ^3, \C)$ s.t.\ $\|\xi(t)-e^{\im t \Delta}\varphi\|_{H^1}+|z(t)|\to 0$ as $t\to +\infty$ (resp.\ $-\infty$).
If $(\xi,z)$ scatters forward and backward in time, we simply say that $(\xi,z)$ scatters.
\end{definition}

Our aim of this paper is to show that all radial solutions with small $\mathbb M$ scatter.
The following is our main result.

\begin{theorem}\label{thm:main}
Assume  $G\in H^1  _{rad}(\R ^3, \C)$ (space of $H^1$ functions depending only on $|x|$), $G\in \mathcal{S}(\R ^3, \C)$ and \eqref{eq:FGR}.
Then, there exists $\delta>0$ s.t.\  if  $\xi(0)\in  H^1  _{rad}(\R ^3, \C)$ satisfies $\|\xi(0)\|_{L^2}+|z(0)|\leq \delta$  the solution $(z,\xi)$ of the system \eqref{1}-\eqref{2} scatters.
\end{theorem}
\noindent
Notice that \eqref{1} is $L^2$ supercritical because the nonlinear term $|\xi|^2\xi$ is $H^{1/2}$ critical.
We  emphasize that we are not in the perturbation regime,  because we are only assuming the smallness of $\mathbb M$ (i.e.\ the $L^2$ norm) while   the $H^1$ norm is arbitrary.
Theorem \ref{thm:main} seems to be the first result of radiation damping in the non--perturbation regime  and the first result with the CCR method applied to the situation where meta--stable torii exist.
%A perturbed treatment of the radiation damping will be given in Theorem \ref{thm:1}, where we assume  $H^{1/2}$ is small.
We remark that there exist   solutions of \eqref{1}--\eqref{2} which do not scatter.
For example all  negative energy  solutions do no scatter.
We do not address here the more general problem of  whether or not
all positive energy solutions of \eqref{1}--\eqref{2} scatter when we assume  \eqref{eq:FGR}.

\begin{remark}
The choice of  $\xi (0)$ and $G$  radial guarantees that $\xi (t,x)$ is radial in $x$. This condition  is important for the profile decomposition  in
Section \ref{sec:profdec} which uses the compactness of $H ^{1}_{rad}\hookrightarrow L^4$.

\end{remark}

\begin{remark}
 In the case of system \eqref{1}--\eqref{2} without the $|\xi |^2 \xi $ term, if \eqref{eq:FGR} is true then the proof and result of  Theorem \ref{thm:1}  hold   without the hypothesis \eqref{eq:smallen}. On the other hand, if   $ \widehat{ G}(y)=0$ for all $y \in \{ y \in\R^3\ |\ 1-\delta<|y|<1+\delta\}$ for some $\delta>0$, then  for sufficiently small $\epsilon>0$, for
\begin{align*}
\omega_\epsilon:=\frac 3 2  \epsilon^4(G| (-\Delta-1-\omega_\epsilon)^{-1}G),
\end{align*}
one can show $(\epsilon^3 e^{-\im(1+\omega_\epsilon)t}(-\Delta-1-\omega_\epsilon)^{-1}G, \epsilon e^{-\im(1+\omega_\epsilon)t} )$ is a family of standing wave solutions.
\end{remark}

As mentioned above, our work is motivated by  Nakanishi \cite{NakanishiJMSJ} which studies
\begin{equation}\label{eq:NLSpot}
 \im \dot u =(-\Delta +V) u +|u|^2u \text{ in $H ^{1}_{rad}(\R^3, \C)$},
\end{equation}
where $V(x)=V(|x|)$, $-\Delta +V$  restricted in  $H ^{1}_{rad}(\R^3, \C)$
has  just one   strictly negative eigenvalue.
%   and a less restrictive condition than  $V\in \mathcal{S}(\R ^3,\R )$.
In  \cite{GNT}  it had been proved that for $ \| u (0)\| _{H^1}\ll 1  $ the  solution  $u(t)$ of  \eqref{eq:NLSpot} can be written   as
\begin{equation*}
 u(t)=   e ^{\im \vartheta (t) }Q_{  \omega _+}  + e^{\im t\Delta }\eta  _+ + o_{H^1}(1)
\end{equation*}
 with $\vartheta \in C^1([0, \infty ), \R )$, $Q_{  \omega _+}$  a  nonlinear ground state (possibly $Q_{  \omega _+}=Q_0:=0$), $\eta _+\in H^1(\R ^3,\C )$ and $o_{H^1}(1)\stackrel{t\to +\infty}{\rightarrow} 0$   in $  H^1(\R ^3,\C )$.
 Nakanishi \cite{NakanishiJMSJ}
 has strengthened the result in  \cite{GNT}    easing the condition $ \| u (0)\| _{H^1}\ll 1  $  by enlarging the basin of attraction into $ \| u (0)\| _{L^2}\ll 1  $
 and  $u (0)\in H^1(\R ^3,\C )$, and by adding also that both $u(0)$ and $V$ are radially symmetric. %Notice that this result is a first step of the SR Conjecture for \eqref{eq:NLSpot} (for a formulation of the conjecture for \eqref{eq:NLSpot}  see Sect.8 \cite{W3}).

%This seems to be the first result of scattering in the non--perturbation regime (or application of CCR method) where the phase space contains a soliton.

\noindent In  \cite{CM15APDE} we extended the result  of  \cite{GNT} analyzing small $H^1(\R ^3,\C )$ solutions   in the case of  $-\Delta +V$  with generic $\sigma _p(-\Delta +V)$
 proving  that, up to scattering and symmetries, a small $H^1(\R ^3,\C )$ solutions
 converges  to a small soliton, perhaps to vacuum.
%Notice that the linear Schr\"odinger equation with potential have quasi-periodic solutions but our result tells there are no quasi-periodic solutions.
%This is due to the radiation damping which separates the nature of solutions of one eigenvalue case and the multiple eigenvalue cases.
 It would be natural, following  \cite{NakanishiJMSJ}, to extend the result in \cite{CM15APDE}
to the case of solutions with small $L^2(\R^3,\C)$ norm but arbitrary $H^1(\R ^3,\C )$  norm. This remains an open problem although the arguments presented in this paper come very close to prove this, as we explain below.

We explain now the main features of the proof of this paper.
Like in \cite{NakanishiJMSJ} the proof is divided in two parts. In the first we perform
a profile decomposition of sequences of solutions and, proceeding by contradiction, we
find a "minimal non-scattering solution". In the last part of the proof we derive a contradiction using    the same argument of \cite{NakanishiJMSJ}.

 In most of the literate, when there is no small localized state, see for example  \cite{gerard,keraani,KM06Invent,holmer},
 an important tool is the existence of
   nonlinear profiles  associated to "concentrating waves" (the latter are the
   waves $\lambda^j_n$ of the expansion
  \eqref{prop:lindecomp1}).
  Key is the existence of wave operators, see  p.\ 50 \cite{strauss}, which allow to
  associate to any solution of the free linear equation a solution of the nonlinear equation  with the same asymptotic behavior (the \textit{nonlinear profile}) as $t\to +\infty$.
   However, in the presence of some discrete coordinate  the existence of wave operators is  a nontrivial problem. In fact, in the context considered by Nakanishi \cite{NakanishiJMSJ}, where there is a small localized solution,  the uniqueness of  the  {nonlinear profile} is unknown and the existence is obtained by weak limit (see \cite{GNT}).
In  situations where radiation damping occurs, such as our system \eqref{1}--\eqref{2}, the situation seems to be the same as in Nakanishi \cite{NakanishiJMSJ}.
That is, although there is no small localized solution, we do not have the uniqueness of the final data problem.
To overcome this difficulty, Nakanishi's ingenious idea in \cite{NakanishiJMSJ} was to define the nonlinear profiles from weak limits and to consider two different nonlinear perturbation estimates, one     close to  the profiles and the other away from   them.
For this purpose Nakanishi introduced a seminorm, here called Nakanishi's seminorm, based on $\st :=L^4_tL^6_x$, to measure the difference of a solution  of the nonlinear problem  from an associated solution of the linearized problem.

\noindent
In this paper, we   follow Nakanishi's strategy.
The difficulty  is that we have  additionally  the forcing term $|z|^2z G$ in \eqref{1} which could derail  Nakanishi's strategy.
  Indeed, in \cite{NakanishiJMSJ} the bootstrap arguments, and with them the whole construction,  are based on the fact that there are no meta--stable torii and
  that the nonlinearity  does not contain forcing terms like $|z|^2z G$, which we instead   consider in this paper and which are essential  for  radiation damping.
In general we  expect  that    any nonlinear dispersive equation with
meta--stable torii and for which it is necessary to prove
radiation damping  displays the difficulties we face in this paper.

%, we will face the difficulty coming from this term (actually, for example in the asymptotic stability of solitons, a large part of the effort is devoted to "find" the forcing term by normal form reduction to obtain the radiation damping).

To find the  minimal non-scattering solution   we consider sequences ${(\xi _n , z_n)}$ of solutions of  \eqref{1}--\eqref{2}.
Following, Nakanishi, it is natural to  try nonlinear profile decompositions
\begin{align}\label{eq:fordecomp}
\xi _n \sim \sum_{j=0}^{J-1}\xi^j(\cdot-s^j_n) + \Gamma^J_n
\end{align}
with    $(\xi^j,z^j)$ satisfying  \eqref{1}--\eqref{2},  scattering  forward and with  the $\xi^j(\cdot-s^j_n)$  localized in  temporal regions $(s^j_n-\tau,s^j_n+\tau)$, for  some $\tau \gg 1$.
Then one has to show that also  $\xi_n$   scatters by showing $\xi_n$ has finite $\st$ norm in $(0,\infty)$.
In this argument the difficulties arise with the remainders $\Gamma^J_n$.
Since we expect the nonlinear remainder to exist essentially  only in the ``gap" region $I^{j,\tau}_n:=(s^j_n+\tau,s^{j+1}_n-\tau)$, we divide $\Gamma^J_n$ into $J$ pieces $\Gamma^{J,j,\tau}_n$.
The nonlinear remainder $\Gamma^{J,j,\tau}_n$ will then be given by the solution of NLS with the forcing term and the nonlinear term restricted on $I^{j,\tau}$ and initial data given by $\Gamma^{J,j,\tau}_n(s^j_n+\tau)=\gamma^J_n(s^j_n+\tau)$, where $\gamma^J_n$ is the remainder of the linear profile decomposition (see \ref{Assumption:A9} and  \eqref{eq:Jjt} in particular).

Since the key in Nakanishi's argument is to show that $\xi_n$ are well approximated by the the nonlinear profiles and the remainders and moreover estimate them by the norm $\st=L^4_tL^6_x$, we need to establish various estimates based on the $\st$ norm.
In particular, we need a
\begin{equation}\label{Fermi Golden rule}
     {\textit{
$L^4_t$  estimate on   the forcing term dependent only on discrete modes,}}
\end{equation}
which is obtained here by an elementary manipulation of the
basic Fermi Golden rule identity, see \eqref{11.11.1}--\eqref{11.12} later.
We remark that although the $L^4_t$ based FGR can be obtained easily in this case, the obstruction to extending \cite{NakanishiJMSJ} to the setting of \cite{CM15APDE} comes from the lack of such estimate.

As in \cite{NakanishiJMSJ}, we will estimate the profiles step by step by moving from $(s^j_n-\tau,s^j_n+\tau)$ to $I^{j,\tau}_n=(s^j_n+\tau,s^{j+1}_n-\tau)$ successively estimating $\xi^j(\cdot-s^j_n)$ and $\Gamma^{J,j,\tau}$.
The difference will come mainly in the ``gap" region $I^{j,\tau}_n$ where the nonlinear remainder $\Gamma^{J,j,\tau}_n$ cannot be estimated a priori and we have to include their estimates in the iterative procedure.
In particular, to be able to estimate $\Gamma^{J,j,\tau}_n$, we need to bound the forcing term $|z_n|^2z_n G$.
To bound the forcing term we have to use the equation of $\xi_n$ and the \eqref{Fermi Golden rule}, and we need to go back to the fact that in the region $(s^j-\tau,s^j+\tau)$, $\xi_n$ is well approximated by $\xi^j(\cdot-s^j_n)$.
In the region $I^{j,\tau}_n$, we will show that $\xi_n$ is well approximated by $\Gamma^{J,j,\tau}_n$ and the remainder $\Gamma^{J,j,\tau}_n$ itself is small in $\st$.
To proceed from $I^{j,\tau}_n$ to $(s^{j+1}_n-\tau,s^{j+1}_n+\tau)$, the key is to show that $\Gamma^{J,j,\tau}_n$ is negligible in this region.
Since $\Gamma^{J,j,\tau}_n$ has no nonlinearity nor forcing term after $s^{j+1}_n-\tau$ (by definition, see \ref{Assumption:A9}), one can show this by Duhamel estimates (see Lemma \ref{lem:referee}) provided $(\xi^{j+1},z^{j+1})$ scatters backward, which implies the forcing term $|z_n|^2z_n G$ will be negligible near $s^{j+1}_n-\tau$.
By such argument, we can estimate the profiles one by one and in the same time show the profiles are good approximation in each regions.

The discussion we made after \eqref{eq:fordecomp}, while framed for system \eqref{1}--\eqref{2}, is in fact very general  and can  be reproduced in the framework of \cite{CM15APDE} or in other settings. The only gap remaining in order to extend  the result of
\cite{CM15APDE} to all solutions with small $L^2$ norm and arbitrary $H^1$ norm consists
in getting \eqref{Fermi Golden rule} which, while easy here, might be nontrivial in the
situation considered in \cite{CM15APDE}.

Another result which  can be proved exploiting the present paper
involves a problem treated in  \cite{ckp} involving
\begin{equation}\label{NLS1}
 \im   \dot u  (t,x)= ( -\Delta +V(x)+\lambda )u (t,x)+(1 + \gamma _1 \cos(t) )|u (t,x)|^2u(t,x) ,
 \, u (0,x)=u_0(x)
\end{equation}
where $\lambda  $ is a  constant.
Specifically we can prove the following result, which we only state here.

\begin{theorem}\label{theorem-1.1}
  Assume that $-\Delta +V$   has exactly one negative eigenvalue given by $-\lambda  $ with
    $0<\lambda < 1$. Assume $0<\gamma _1<1$. Assume the hypotheses stated in Theorem \cite{ckp} and the $V$ is a radial Schwartz function.
    Then, there exists an $\epsilon_0>0$
  s.t.\  if $  \|u_0  \|_{L^2}<\epsilon_0  $
  and $u_0\in H^1_{rad} (\R ^3, \C )$ there exists a  $\varphi \in H^1(\R ^3, \C)$ s.t.\
\begin{equation}\label{scattering}
\lim_{t\to  +\infty}\|u(t ) -e^{\im t\Delta } \varphi \|_{H^1}=0
.
\end{equation}
 \end{theorem}
In   \cite{ckp}   the above result was  proved with $  \|u_0  \|_{H^1}<\epsilon_0  $.
Here the restriction  $\gamma _1<1 $, not present in \cite{ckp},  is added to allow any value of the norm
$  \|u_0  \|_{H^1}.$  Notice that the result in \cite{ckp}  was extended by  \cite{Cu3} to the case when
$-\Delta +V$ has any number of eigenvalues in  $( -\lambda , 0)$, but that Theorem \ref{theorem-1.1} is stated only when $-\Delta +V$   has exactly one   eigenvalue exactly because only in this case  we can get   \eqref{Fermi Golden rule}. Indeed, in analogy with
\eqref{11.11.1}--\eqref{11.12}, the desired bound  can be obtained in an elementary fashion
by considering   formula (4.23) (in the case $n=0$) in \cite{Cu3}
   \begin{equation*}
\begin{aligned} &  \frac{1}{2} \frac{d}{dt}|\zeta _0|^2
+\pi  |\zeta _0|^6
    \langle \delta (-\Delta +V +\lambda -1)
  \overline{\mathbf{\Phi} }, \mathbf{\Phi} \rangle =
  \Im \left (  \mathcal{D} _0 \overline{\zeta} _0\right )
\end{aligned}
\end{equation*}
where $\mathbf{\Phi}(x)$ is a rapidly decreasing and $C^2$ function and where
the r.h.s.\ is a remainder term. Then multiplying the formula by   $|\zeta _0|^6$ and proceeding in a fashion similar to \eqref{11.11.1}--\eqref{11.12} we get an estimate on $\| \zeta _0  \| _{L ^{12}_t} $.
In the presence of one or more further discrete modes
we don't know yet how to get  \eqref{Fermi Golden rule}.

\noindent

Provided that we can get an $\st$ bound on appropriate discrete components interactions
our strategy can be applied on a diverse set of problems.  The result in Theorem \ref{thm:main}, or the result in Theorem \ref{theorem-1.1}, are somewhat restricted to very
special classes of systems. However the method we develop in this paper, combined with some other ingredients involving the Fermi Golden rule, promises to be relevant in much more general situations.  We think that
the approach to   the Soliton Resolution which is currently taking shape, will need ultimately to face the problems we consider in the present paper, and possibly borrow some of the ideas we present here in the presence of {  metastable   torii}.

Finally a few words on the organization of the paper.
In section \ref{section:set up}, we prepare notations and give a proof of Theorem \ref{thm:main} under the more restrictive condition $\|\xi(0)\|_{H^{1/2}}+|z(0)|\ll1$.
In section \ref{sec:lin_est} we collect the known linear estimates and introduce Nakanishi's seminorm.
In section \ref{sec:l4}, we provide the $L^4$ in time estimates.
In section \ref{sec:np}, we prove nonlinear perturbation estimates.
In section \ref{sec:profdec}, we give the linear and nonlinear profile decomposition.
In section \ref{sec:it}, we perform the main iteration argument and in section \ref{sec:scattering}, we show the scattering and complete the proof.

   \section{Notation and preliminary results} \label{section:set up}

We will use the following standard notation.
\begin{itemize}

\item
$L^{2,s}(\R^3, \C ):=\{u\in \mathcal S'(\R^3, \C ) |\ \<x\>^s u \in L^2(\R^3, \C ) \}$ with $\mathcal S'(\R^3, \C )$ the space of tempered distributions and $\<x\> = \sqrt{1+|x|^2}$.

\item  $B^s_{p,q}(\R^3, \C )$  is the Besov space formed by the   tempered
distributions $f\in {\mathcal S}'(\R^3,\C )$  s.t.\
\[
   \|f\|_{B^s_{p,q}}=  (\sum_{j\in\N}2^{jsq} \|\varphi_j *
f\| _{L^p(\R ^3)}^q )^{\frac{1}{q}}<+\infty
\]
with~$\widehat{\varphi} \in {  C}^\infty_c(\R^3\setminus
\left\{0\right\})$ s.t.\ ~$\sum_{j\in
\Z}\widehat{\varphi}(2^{-j}\xi)=1$ for all~$\xi
\in\R^3\setminus\left\{0\right\}$,
$\widehat{\varphi}_j(\xi)=\widehat{\varphi}(2^{-j} \xi)$ for all
$j\in\N^*$ and for all~$\xi \in\R^3$, and
$\widehat{\varphi_0}=1-\sum_{j\in \N^*}\widehat{\varphi}_j$.

\item We will simplify the notation and write $L^{2,s}$  for  $L^{2,s}(\R^3, \C )$, $H^s$ for $H^s(\R^3, \C )$, $B^s_{p,q}$ for $B^s_{p,q}(\R^3, \C )$, $L^p$ for $L^p (\R^3, \C )$, $\mathcal{S}$ for $\mathcal{S}(\R^3, \C )$ and $\mathcal{S}'$ for $\mathcal{S}'(\R^3, \C )$.

\item Given an interval $I\subseteq \R$  and a Banach space $\mathbb{X}$ we set  $L^p\mathbb{X} (I) := L^p(I, \mathbb{X}) $.

\item Given an interval $I\subseteq \R$ we set
$\mathrm{Stz}^s(I):=L^\infty  H^s (I)   \cap L^2  B^s_{6,2}(I),$ $ \mathrm{Stz}^{*s}(I):=L^1  H^s(I )  +  L^2 B^s_{6/5,2}(I )$ and $
\mathfrak{st}(I):=L^4 L^6(I).$
\item
We set $\displaystyle R_+(1)=\lim_{\varepsilon\to 0 ^{+}}(-\Delta-1-\im \varepsilon)^{-1}$ which    for $\sigma>1/2$ exists in the strong sense
  in the space $B (L^{2,\sigma}, L^{2,-\sigma})$     of bounded linear operators $L^{2,\sigma}\to L^{2,-\sigma}$.
 % \item     We set
%  \begin{equation}\label{eq:defHt}
 % \text{$\mathbb{H}_{\theta} (u)   := 2^{-1}  \<(-\Delta)^\theta u ,u\>$  for $\theta  \ge 0$.}
  %\end{equation}

\item We write that $0\le a\ll 1$ if $0\le a\le \epsilon $ for  a preassigned and arbitrarily small $\epsilon>0$.

\item We write $a\lesssim   b$   if $a\le C b$ for a   preassigned   $C>0$.

\end{itemize}
We set
\begin{align}\label{def:Gamma}
\Gamma =-\Im \beta,\ \beta=(G|R_+(1)G).
\end{align}
By $R_+(1)=P.V.\frac{1}{-\Delta-1}+\im \pi \delta (-\Delta-1)$, we have
\begin{align}\label{10.5}
\Gamma =\pi (G| \delta(-\Delta-1)G)=\pi \int_{|\eta|=1}|\hat G (\eta)|^2\,d\eta\geq  0.
\end{align}
Under the assumption \eqref{eq:FGR} we have $\Gamma>0$.

The system \eqref{1}--\eqref{2}  satisfies the following (easier)  analogue of the main result in \cite{GNT,CM15APDE}.

\begin{theorem}\label{thm:1}
Assume \eqref{eq:FGR}.
Then, there exist  $\delta>0$  and $C>0$  s.t.\ if $\xi(0)\in H^1$ and
\begin{equation}\label{eq:smallen}
\|\xi(0)\|_{\Hh}+|z(0)|\leq \delta
\end{equation}
 we have
\begin{align}\label{4.0}
\|\xi\|_{\stzt(\R)}+\|z\|_{L^6(\R)}^3+ \|z\|_{L^\infty (\R)}\le C   (\|\xi(0)\|_{\Ht}+  \|\xi(0)\|_{L^2}+   |z(0)|),
\end{align}
for $\theta\in [0,1]$.
In particular, $( \xi ,z )$ scatters.
\end{theorem}

\begin{proof}
By the Strichartz estimates (Lemma \ref{lem:1}), for $\theta\in [0,1]$, we have
\begin{align}\label{4.1}
\|\xi\|_{\stzt(t_0,t_1)}\le C \left ( \|\xi(t_0)\|_{\Ht}+\|\xi\|_{\st(t_0,t_1)}^2\|\xi\|_{L^\infty H^\theta (t_0,t_1)}+\|z\|_{L^6(t_0,t_1)}^3  \right ) .
\end{align}
We set
\begin{align}\label{5}
Y=-|z|^2z R_+(1)G,\quad R_\pm(z):=(-\Delta-z\mp \im 0)^{-1}
\end{align}
and $\xi = Y+g$.
Then, $g$ satisfies
\begin{align}\label{6}
\im \dot g = -\Delta g + |\xi|^2\xi +\mathcal R_1,\quad \mathcal R_1:=-\im Y_t -\Delta Y + |z|^2z  G.
\end{align}
Substituting \eqref{2} and \eqref{5} into the definition of $\mathcal R_1$, we have
\begin{align}\label{7}
\mathcal R_1= \frac 3 2 |z|^4 \overline{(G|\xi)}R_+(1)G.
\end{align}
Thus, by Strichartz estimate (Lemma \ref{lem:1}) and Lemma \ref{lem:7} below, for $\sigma >9/2$, we have
\begin{equation}
\begin{aligned}
     \|g\|_{L^2 L^{2,-\sigma}(t_0,t_1)}& \le C\|e^{\im \Delta (t-t_0)}g(t_0)\|_{L^2L^{2,-\sigma}(t_0,t_1)}\\& + C \|\xi\|_{\st(t_0,t_1)}^2   \| \xi\|_{L^\infty H ^{\theta}(t_0,t_1)} + C\|z\|_{L^\infty(t_0,t_1)}  ^{4} \| \xi\|_{L^2L^6 (t_0,t_1)}.
\end{aligned}\nonumber
\end{equation}
By Lemma \ref{lem:7} below we have
\begin{align*}
\|e^{\im \Delta (t-t_0)}g(t_0)\|_{L^2L^{2,-\sigma}(t_0,t_1)}&\le C'  \|\xi(t_0)\|_{L^2}+|z(t_0)|^3\|e^{\im \Delta (t-t_0)}R_+(1)G\|_{L^2 L^{2,-\sigma}(t_0,t_1)}\\&\le C\left (  \|\xi(t_0)\|_{L^2}+|z(t_0)|^3  \right ) .
\end{align*}
Thus, there exists a fixed constant $C$  s.t.\
\begin{equation}\label{8}
\begin{aligned}
    \|g\|_{L^2 L^{2,-\sigma}(t_0,t_1)}\le & C (\|\xi(t_0)\|_{L^2}+|z(t_0)|^3) \\& + C \|\xi\|_{\st(t_0,t_1)}^2   \| \xi\|_{L^\infty H ^{\theta}(t_0,t_1)} + C\|z\|_{L^\infty(t_0,t_1)}  ^{4} \| \xi\|_{L^2L^6 (t_0,t_1)}.
\end{aligned}
\end{equation}
Substituting $\xi=Y+g$ into \eqref{2}, we obtain
\begin{align}\label{9}
\im \dot z & = z + \frac 1 2 z^2 (G|g) +|z|^2\overline{(G|g)}-z|z|^4\(\frac 1 2 (G|R_+(1)G)+\overline{(G|R_+(1)G)}\).
\end{align}
Thus, multiplying $\bar z$ and taking the imaginary part we have
\begin{align}\label{10}
\frac{d}{dt}|z(t)|^2 =-\Gamma\frac{1}{2}|z|^6 + \Im\(\frac 1 2 |z|^2z  (G|g) +|z|^2\bar z\overline{(G|g)}\),
\end{align}
where $\Gamma$ is given by \eqref{def:Gamma}.
Thus\begin{align}\label{eq:z^6}
\Gamma \|z\|_{L^6(t_0,t)}^6 + 2|z(t)|^2&\le  2 |z(t_0)|^2 +   3\|G\|_{ L^{2, \sigma} }    \|g\|_{L^2L^{2,-\sigma}(t_0,t_1)}\|z\|_{L^6(t_0,t)}^3,
\end{align}
for $t_0<t<t_1$.
 Taking $\sup_{t_0<t<t_1}$, we have for fixed constants
\begin{equation} \label{11}
\begin{aligned} & \Gamma^{1/2}\|z\|_{L^6(t_0,t_1)}^3 +  \|z\|_{L^\infty (t_0,t_1)} \le C' ( |z(t_0)| + \Gamma^{-1/2}\|g\|_{L^2L^{2,-\sigma}(t_0,t_1)})\\&
\le C  |z(t_0)| +C \Gamma^{-1/2} \( |z(t_0)| ^3 +  \|\xi(t_0)\|_{L^2}  +  (\|\xi\|_{\st(t_0,t_1)}^2   + \|z\|_{L^\infty(t_0,t_1)}  ^{4}) \| \xi\|_{\stz ^{\theta} (t_0,t_1)}\).
\end{aligned}
\end{equation}
Substituting \eqref{11} into \eqref{4.1}, we have
\begin{equation} \label{11.0} \begin{aligned}    \|\xi\|_{\stzt(t_0,t_1)}&\le C  (\|\xi(t_0)\|_{\Ht} +   \Gamma^{-1/2} |z(t_0)|+  \Gamma^{-1} |z(t_0)|^3) \\& +  C \(   \Gamma^{-1}\|z\|_{L^\infty(t_0,t_1)}^4 + \<\Gamma^{-1}\>\|\xi\|_{\st(t_0,t_1)}^2\)\| \xi\|_{\stz ^{\theta} (t_0,t_1)}.
\end{aligned}
\end{equation}
The estimate \eqref{4.0} with $\theta\leq 1/2$ follows from a simple continuity argument combined with the smallness of $\|\xi(t_0)\|_{\Ht}+|z(t_0)|$  for $t_0=0$.
For $\theta\in (1/2,1]$, \eqref{4.0} follows from \eqref{4.1} combined with $\stzh\hookrightarrow \st$.

\noindent Finally, we show  scattering, which is a simple consequence of \eqref{4.0}.
Since
\begin{align*}
e^{-\im t_2 \Delta}\xi(t_2)-e^{-\im t_1 \Delta}\xi(t_1)=-\im \int_{t_1}^{t_2} e^{-\im s \Delta}\(|\xi(s)|^2\xi(s)+z(s)|z(s)|^2G\)\,ds,
\end{align*}
it suffices to show
\begin{align*}
\|\int_t ^{+\infty }e^{-\im s \Delta}\(|\xi(s)|^2\xi(s)+z(s)|z(s)|^2G\)\,ds \|_{H^1}\to 0,\quad\text{as }t\to +\infty
\end{align*}
From Lemma \ref{lem:1}  the above integral is bounded by $\|\xi\|_{\stzo(t,\infty)}^3+\|z\|_{L^6(t,\infty)}^3 \stackrel{t\to +\infty}{\rightarrow}0$.
In addition, since $\|\im \dot z-z\|_{L^\infty (\R _+)}\lesssim \|z\|_{L^\infty (\R _+)}^2\|\xi\|_{L^\infty L^2 (\R _+)}$
we conclude from \eqref{4.0} a bound on $\|z\|_{L^\infty (\R _+)}$. Hence from
$\|z\|_{L^\infty (t,\infty)}\lesssim \|\dot z\|_{L^\infty(t,\infty)}^{1/7}\|z\|_{L^6(t,\infty)}^{6/7}
$  we obtain $|z(t)| \stackrel{t\to +\infty}{\rightarrow}0$.
This gives forward scattering, and since it is possible to prove backward scattering by the same argument, the proof of
Theorem \ref{thm:1}  is complete.
\end{proof}

\begin{remark}\label{rem:1}
The conclusion about forward scattering  of Theorem \ref{thm:1}  continues to hold  if we replace the small energy hypothesis \eqref{eq:smallen}
with the hypothesis   $\|\xi\|_{\stzo(0,\infty)}<\infty$.
Indeed, by \eqref{11} we have $z\in L^6(0,\infty)$ and therefore the   argument  at the end of the proof of Threorem \ref{thm:1} can be repeated.
\end{remark}

We have the following preliminary result, based uniquely of the conservation of $\mathbb E$ and $\mathbb M$.

\begin{lemma}\label{lem:-8}
Let $( \xi ,z)$ be the solution of \eqref{1}--\eqref{2} with $ \xi(t_0) \in   H^1$.
Assume $\cN_0  \lesssim  1$
 for
\begin{equation}\label{eq:preprop11}
 \cNt:=\|\xi (t_0)\|_{  \Ht }+ |z (t_0) | .
\end{equation}
Then there exist     $C_0 =C(\cN_0)$  s.t.\
\begin{align}      & \|\xi\|_{L^\infty H ^{1}(\R )}+\|z\|_{L^\infty(\R )} \le C_0 \   (\mathcal{N}_1  + \mathcal{N}_1^2). \label{11.9.-00}
\end{align}
\end{lemma}
\proof From the conservation of $\mathbb{E}$ and $\mathbb{M}$, see \eqref{4}--\eqref{eq:mass},
we have
\begin{equation*}     \begin{aligned} & \|  \xi \|_{L^\infty H^1(t_0,t)} + \|z\| _{L^\infty (t_0,t)} \lesssim \mathcal{N}_1 + \mathcal{N}_0  ^{\frac{1}{2}}\mathcal{N}_1^{\frac{3}{2}}+ \mathcal{N}_0  ^{2}  + \|z\| _{L^\infty (t_0,t)}^{\frac{3}{2}} \|  \xi \|_{L^\infty L^2(t_0,t)}^{\frac{1}{2}}
\end{aligned}\end{equation*}
which  by $\|z\| _{L^\infty (t_0,t)}+ \|  \xi \|_{L^\infty L^2(t_0,t)}  \le  2 {\mathcal{N}_0}$ due to the conservation of $\mathbb{M}$,
by $\cN_0\lesssim  1$ and  by $\mathcal{N}_0\le \mathcal{N}_1$   implies immediately \eqref{11.9.-00}.
\qed

\section{Linear estimates}\label{sec:lin_est}

In this section  we set some notation and   list   estimates  about the linear Schr\"odinger
equation which are used    in   subsequent sections.
We use material from section 4 of \cite{NakanishiJMSJ}.
For $u\in C(\R;H^1)$, we set
\begin{align}
u[t_0](t)&:=e^{\im (t-t_0)\Delta}u(t_0),\label{11.1}\\
u[t_0]_>(t)&:=\begin{cases}u(t),& \text{if }t\leq t_0,\\ u[t_0](t),& \text{if }t> t_0.\end{cases}\label{11.2}
\end{align}
For $u_0\in H^1$, we identify $u_0$ with $u(t)\equiv u_0$ and define
 \begin{equation}\label{eq:u0t0}
 u_0[t_0](t)=e^{\im (t-t_0)\Delta}u_0.
 \end{equation}
%\begin{remark}
%If we write $u(t_1)[t_2](t_3)$, this will mean $e^{\im (t_3-t_2)\Delta}u(t_1)$.
%Generally, this is different from $u[t_2](t_3)=e^{\im (t_3-t_2)}u(t_2)$ if $t_1\neq t_2$.
%\end{remark}
The solution of
\begin{align*}
\im \dot v = -\Delta v + f,\quad v(t_0)=0
\end{align*}
can be written  as
\begin{align}\label{11.3}
\mathcal Df[t_0](t):=-\im \int_{t_0}^t f[s](t)\,ds.
\end{align}
We   can  express as
$
u_0[t_0]+\cD f[t_0]
$  the  solution of
\begin{align*}
\im \dot v = -\Delta v + f,\quad v(t_0)=u_0.
\end{align*}

\begin{remark}
We have
\begin{itemize}
\item
$u[t_1][t_2](t)=e^{\im (t-t_2)\Delta}u[t_1](t_2)=e^{\im (t-t_2)\Delta} e^{\im (t_2-t_1)\Delta}u(t_1)=e^{\im (t-t_1)\Delta}u(t_1)=u[t_1](t)$.
\item
${\(\cD f[t_1]\)[t_2](t)=-\im \int_{t_1}^{t_2} f[s](t)\,ds}$.
\end{itemize}
\end{remark}

The following are the classical Strichartz estimates, see Theorem 2.3.3 \cite{caz}.
\begin{lemma}[Strichartz estimates]\label{lem:1}
There   exist constants   $C _{\theta}$   s.t.\   for any interval $I \subseteq \R$ with $t_0\in I$ and any $f$
\begin{align*}
\| u_0[t_0]\|_{\stzt(I)}&\le C _{\theta} \|u_0\|_{H^\theta},\\
\sup _{t\in \R }\|\int_{I}  f[s](t )\,ds\|_{H^\theta}&\le C _{\theta}  \|f\|_{\stz^{*\theta}(I)},\\
\|\cD f[t_0]\|_{\stzt(I)}&\le C _{\theta}  \|f\|_{\stz^{*\theta}(I)}.
\end{align*}
\end{lemma}

%\begin{lemma}\label{lem:1.1} Let $z\in C\cap L^\infty(\R;\C^n)$ with $\|z\|_{L^\infty}\ll1$. Then, we have \begin{align*} \|u_0\|_{\dot H^\theta}\sim \inf_{t}\|u_0[t_0]\|_{\dot H^\theta}\sim \sup_{t}\|u_0[t_0]\|_{\dot H^\theta}. \end{align*} \end{lemma} \begin{proof} First, it is immediate from Lemma \ref{lem:1} that \begin{align*} \inf_{t}\|u_0[t_0]\|_{\dot H^\theta}\lesssim \sup_{t}\|U[z](\cdot, t_0)u_0\|_{\dot H^\theta}\lesssim\|u_0\|_{\dot H^\theta}. \end{align*} Next, take $\{t_n\}_n$ s.t.\ $\|u_0[t_n](t_0)\|_{\dot H^\theta}\to \inf_{t}\|u_0[t_0]\|_{\dot H^\theta}$. Then, since \begin{align*} \|u_0\|_{\dot H^\theta}=\|\(u_0[t_0]\)[t_n](t_0)\|_{\dot H^\theta}\lesssim \|\(u_0[t_n]\)(t_0)\|_{\dot H^\theta}\to \inf_{t}\|u_0[t_0]\|_{\dot H^\theta}, \end{align*} we have the conclusion. \end{proof}

The following estimates are due to Kato \cite{kato}, Foschi \cite{foschi} and Vilela \cite{vilela}.
\begin{lemma}[Non-admissible Strichartz]\label{lem:2} Let
\begin{align}\label{11.5}
(p_j,q_j)\in (1,\infty)\times (2,6]\  (j=1,2)\text{ and }\sigma_j:=\frac{2}{p_j}+3\(\frac{1}{q_j}-\frac 1 2\)
\end{align}
satisfy
\begin{align}\label{11.6}
\sigma_0+\sigma_1=0>\sigma_j-\frac{1}{p_j},\quad |\sigma_j|\leq 2/3.
\end{align}
Then   there exists  a constant $C$    s.t.\ for any interval $I$ with $t_0\in I$ and any $f$
\begin{align*}
\|\cD f[t_0]\|_{L^{p_0}L^{q_0}(I)}\le C \|f\|_{L^{p_1'}L^{q_1'}(I)}.
\end{align*}
\end{lemma}

%
%\begin{remark}
%If we want   $(p_0,q_0)$ and $(p_1=p_0',q_1)$ which satisfies \eqref{11.6}, then we have to have $q_0=q_1=6$. Otherwise, if $q_0<6$, we will need $q_1>6$.
%\end{remark}

We further introduce Nakanishi's seminorm:
\begin{align}\label{11.7}
\|u\|_{\nnorm{T_0}{T_1}}:=\sup_{T_0<S<T<T_1}\|u[T]_>-u[S]\|_{\mathfrak{st}(T_0,\infty)}.
\end{align}

%\begin{remark}
%This is the definition of Nakanishi's seminorm in the first version of \cite{NakanishiJMSJ}, see \cite{N}.
%In the second version, it has been modified as $\sup_{T_0<S<T<T_1}\|u[T]_>-u[S]\|_{\mathfrak{st}(T_0,\infty)}$.
%However, with this choice not know how to prove the nonlinear perturbation estimate Lemma \ref{lem:17}.
%\end{remark}

\begin{remark}
If we take $S=T_0$ and $T=T_1$ then restricting the interval to $(T_0,T_1)$ we have
\begin{align*}
\|u[T_1]_>-u[T_0]\|_{\mathfrak{st}(T_0,\infty)}\geq \|u-u[T_0]\|_{\mathfrak{st}(T_0,T_1)}.
\end{align*}
Similarly, restricting the interval to $(T_1,\infty)$, we have
\begin{align*}
\|u[T_1]_>-u[T_0]\|_{\mathfrak{st}(T_0,\infty)}\geq \|u[T_1]-u[T_0]\|_{\mathfrak{st}(T_1,\infty)}.
\end{align*}
Therefore, we have
\begin{align}
\|u\|_{\nnorm{T_0}{T_1} }\geq \max\(\|u-u[T_0]\|_{\mathfrak{st}(T_0,T_1)}, \|u[T_1]-u[T_0]\|_{\mathfrak{st}(T_1,\infty)}\).\label{11.8}
\end{align}
This inequality will be used frequently.
\end{remark}

%\begin{remark} We need to take the $\sup$ over $[T_0,T_1]$ to have the subadditive property below.\end{remark}

Nakanishi's seminorm is dominated by Strichartz's norm.
\begin{align}\label{11.8.1}
\|u\|_{\nnorm{t_0}{t_1} }\leq C\|u\|_{\stz^1(t_0,t_1)}.
\end{align}
Indeed, for $t_0<s<t<t_1$,
\begin{align*}
\|u[t]_>-u[s]\|_{\st(t_0,\infty)}&\le \|u -u[s]\|_{\st(t_0,t)}+\|u[t]-u[s]\|_{\st(t_0,\infty)}\\&\leq C\|u\|_{\stz^1(t_0,t)}+C\|u(s)\|_{H^1}+\|u(t)\|_{H^1}\leq C\|u\|_{\stz^1(t_0,t_1)}.
\end{align*}
We have the following, see Lemma 4.2 of \cite{NakanishiJMSJ}.
\begin{lemma}[Subadditivity]\label{lem:3}
For $T_0<T_1<T_2$,
\begin{align*}
\|u\|_{\nnorm{T_0}{T_2}}\leq \|u\|_{\nnorm{T_0}{T_1} }+\|u\|_{\nnorm{T_1}{T_2}}.
\end{align*}
\end{lemma}

For $  \chi _{(-\infty , T]} $   being the characteristic function of   $(-\infty , T]$ we have the  following   elementary  lemma.
\begin{lemma}\label{lem:4}
For $u=u_0[t_0]+\cD f[t_0]$   we have
\begin{align*}
u[T]_>-u[S]=\cD \( \chi _{(-\infty , T]} f\)[S].
\end{align*}
\end{lemma}

\begin{remark}
By the above lemma, we see that if $u$ is the solution of the inhomogeneous problem, we have
\begin{align}\label{11.9}
\|u\|_{\nnorm{T_0}{T_1} }=\sup_{T_0<S<T<T_1}\|\cD  \chi _{(-\infty , T]}f[S]\|_{\mathfrak{st}(T_0,\infty)}.
\end{align}
\end{remark}

\begin{lemma}\label{lem:5}%[Lemma 4.3]
There is a fixed constant $C$ s.t.\
$\|\cD f[t_0]\|_{\nnorm{T_0}{T_1} }\le C \|f\|_{\mathrm{Stz}^{*1/2}(T_0,T_1)}$.
\end{lemma}

\begin{proof}
By \eqref{11.9}, we have for fixed constants $C'$ and $C$
\begin{align*}
\|\cD f[t_0]\|_{\nnorm{T_0}{T_1}}&=\sup_{T_0<S<T<T_1}\|\cD  \chi _{(-\infty , T]}f[S]\|_{\mathfrak{st}(T_0,\infty)}\\&
\le C' \sup_{T_0<S<T<T_1}\|\cD  \chi _{(-\infty , T]}f[S]\|_{\mathrm{Stz^{1/2}}(T_0,\infty)}
\\&\le C \sup_{T_0<T<T_1}\|  \chi _{(-\infty , T]} f\|_{\mathrm{Stz^{*1/2}}(T_0,\infty)}\leq C \|f\|_{\mathrm{Stz}^{*1/2}(T_0,T_1)},
\end{align*}
with  the embedding $\mathrm{Stz}^{ \frac{1}{2}}\hookrightarrow \mathfrak{st}$ in the 2nd line and  Stricharz estimates (Lemma \ref{lem:1}) in  the 3rd.
\end{proof}

\begin{lemma}\label{lem:6}
Let $(p_1,q_1)$ satisfy $\frac 1 2 =\sigma_1 = 2/p_1+3(1/q_1-1/2)$ and $p_1<2$.
Then for a fixed $C$
\begin{align*}
\|\cD f[t_0]\|_{\nnorm{T_0}{T_1} }\le C \|f\|_{L^{p_1'}L^{q_1'}(T_0,T_1)}.
\end{align*}
\end{lemma}

\begin{remark}
Lemma \ref{lem:6} is an application of the non-admissible Strichartz (Lemma \ref{lem:2}) with $(p_0,q_0)=(4,6)$.
In this case $\sigma_0=-1/2$ and the condition   \eqref{11.5} is equivalent to $\sigma_1=1/2$ and $p_1<2$.
\end{remark}

\begin{proof}
By Lemma \ref{lem:4} applied to   $u=\cD f[t_0]$ in the 1st line and by Lemma  \ref{lem:2} in 2nd line
\begin{align*}
\|\cD f[t_0]\|_{\nnorm{T_0}{T_1}  }&=\sup_{T_0<S<T<T_1}\|\cD \chi _{(-\infty , T]}f [S]\|_{L^4L^6(T_0,\infty)}\\&\le C \sup_{T_0<T<T_1}\|\chi _{(-\infty , T]}f\|_{L^{p_1'}L^{q_1'}(T_0,\infty)} \le C \|f\|_{L^{p_1'}L^{q_1'}(T_0,T_1)}.
\end{align*}
\end{proof}

The following is well known, for a reference see \cite{CM15APDE} Lemma 6.5 (where $\sigma _0=9/2$).
\begin{lemma}\label{lem:7}
There exists $\sigma_0>0$  s.t.\   for any $\sigma>\sigma_0$
the following facts are true: \begin{enumerate}
                                \item we have $R_+(1)\in B(L^{2,\sigma}, L^{2,-\sigma})$ ;
                                \item there exists a constant $C_{\sigma}$ s.t.\  for $v\in L^{2,\sigma}$, we have
\begin{align*}
\|R_+(1)v[0](t)\|_{L^{2,-\sigma}}\le C _ \sigma\<t\>^{-3/2}\|v\|_{L^{2,\sigma}}\text{  for all $t\ge 0$;}
\end{align*}
                                \item for all $p\geq 1$  there is a constant $C_{p,\sigma}$ s.t.\  for $v\in L^{2,\sigma}$
\begin{align*}
\|R_+(1)v[0]\|_{L^pL^{2,-\sigma} (\R _+ )}\le C_{p,\sigma}  \|v\|_{L^{2,\sigma}} .
\end{align*}
                              \end{enumerate}

\end{lemma}

We will need the following Duhamel estimates too.

\begin{lemma}\label{lem:referee}
There is a $C>0$ s.t.\  for any $T>0$ and any  $f\in L^2L^{6/5}(-\infty,-T)$ we have
\begin{align*}
\|\int_{-\infty}^{-T}f[s](t)\,ds \|_{\st(0,\infty)}\le C T^{-1/4}\|f\|_{L^2L^{6/5}(-\infty,-T)}.
\end{align*}
\end{lemma}

\begin{proof}
By $L^{6/5}$-$L^6$ decay estimate we have
\begin{align*}
\|\int_{-\infty}^{-T}f[s](t)\,ds\|_{L^6_x}\lesssim \int_{-\infty}^{-T}|t-s|^{-1} \|f(s)\|_{L^{6/5}}\,ds.
\end{align*}
Thus,
\begin{align*}
\|\int_{-\infty}^{-T}f[s](t)\,ds\|_{\st(0,\infty)}&\lesssim \|\(\int_{-\infty}^{-T}|t-s|^{-2}\,ds\)^{1/2}\|_{L^4(0,\infty)} \|f(s)\|_{L^2L^{6/5}(-\infty,-T)}\\&\leq T^{-1/4}\|f\|_{L^2L^{6/5}(-\infty,-T)}.
\end{align*}
\end{proof}

\section{$L^4$ estimates}\label{sec:l4}
In this section, we estimate the solutions of \eqref{1}--\eqref{2} in terms of the $L^4$ in time based norms   $\|\xi\|_\st$ and $\|z\|_{L^{12}}$ (which can be thought as the $L^4$ norm of $|z|^2z$).
In principle  $\|z\|_{L^{12}}   ^{2}\le \|z\|_{L^6} \|z\|_{L^\infty }$.
However, in some situations $\|z\|_{L^{12}}$ is  small when $\|z\|_{L^6}$ is not  as small as we need.

\begin{lemma}\label{lem:8}
There exist constants $\mu _0>0$,  $\mu  _{1/2}>0$  and $C_{\theta}$ for any $\theta \in [0,1]$ s.t.\
for any solution  $( \xi ,z)$   of \eqref{1}--\eqref{2}  in $ H^1\times \C$ with $\cN_0\le \mu _0$ and
\begin{equation}\label{eq:preprop10}
 \|\xi\|_{\st(t_0,t_1)} \max \{ 1, \cNh^3 \}\le \mu _{1/2},
\end{equation}
 for  $\cNt$  defined by \eqref{eq:preprop11},  we have
\begin{align}      & \|\xi\|_{\stzt(t_0,t_1)}+\|z\|_{L^\infty(t_0,t_1)}+ \|z\|_{L ^{6}(t_0,t_1)} ^3 \le C_{\theta } \cNt . \label{11.9.-0}
\end{align}

\end{lemma}
\proof
By \eqref{4.1} we have
\begin{equation}\label{11.9.-0 pre1}
 \|\xi\|_{\stzt(t_0,t_1)}\lesssim \|\xi(t_0)\|_{\Ht}+ \|z\|_{L ^{6}(t_0,t_1)} ^3.
\end{equation}
Therefore, it suffices to show
\begin{align}\label{11.9.-0 pre2}
\|z\|_{L^\infty(t_0,t_1)}+ \|z\|_{L ^{6}(t_0,t_1)} ^3 \lesssim  \cN_0.
\end{align}
Proceeding like in the proof of Theorem \ref{thm:1}, for $\sigma>9/4$, we write
\begin{equation}
\begin{aligned}
     \|g\|_{L^2 L^{2,-\sigma}(t_0,t_1)}& \le C (\|\xi(t_0)\|_{L^2}+|z(t_0)|^3)\\& + C \|\xi\|_{\st(t_0,t_1)}^2   \| \xi\|_{L^\infty L^2(t_0,t_1)} + C\|z\|_{L^\infty(t_0,t_1)}  ^{\frac{5}{2}} \|z\|_{L^6(t_0,t_1)}   ^{\frac{3}{2}}   \| \xi\|_{\st (t_0,t_1)}.
\end{aligned}\nonumber
\end{equation}
   Inserting this in the inequality  \eqref{eq:z^6}  we obtain
     \begin{equation}
\begin{aligned} & \|z\|_{L^6(t_0,t_1)}^3 +  \|z\|_{L^\infty (t_0,t_1)}
\le C  |z(t_0)|   + C \|\xi(t_0)\|_{L^2} \\& + C \|\xi\|_{\st(t_0,t_1)}^2   \| \xi\|_{L^\infty L^2(t_0,t_1)} +C\|z\|_{L^\infty(t_0,t_1)}  ^{5}    \| \xi\|_{\st (t_0,t_1)}^{2}.
\end{aligned}\nonumber
\end{equation}
 Using this inequality and  \eqref{11.9.-0 pre1}, we derive immediately \eqref{11.9.-0 pre2}.
\qed

The following proposition is the main $L^4$ estimate in this section.

\begin{proposition}\label{prop:1}  There exist constants $\mu _0>0$, $\mu _{1/2}>0$ and $C>0$ s.t.\ for
 any solution  $( \xi ,z)$   of \eqref{1}--\eqref{2} in  $   H^1\times \C$ with
  $\cN_0\le \mu _0 $ which satisfies  either  \eqref{eq:preprop10}
    or
  \begin{equation}\label{eq:preprop20}
 \(\|\xi[t_0]\|_{\st(t_0,t_1)}+|z(t_0)|^2\)\max \{ 1, \cNh^3 \} \le \mu _{1/2},
\end{equation}
we have
\begin{align}   &
\|\xi\|_{\st(t_0,t_1)}+\|z\|_{L^\infty(t_0,t_1)}^2+ \|z\|_{L^{12}(t_0,t_1)}^3\le C  ( \|\xi[t_0]\|_{\st(t_0,t_1)}+ |z(t_0)|^2),\label{11.9.0}
\\ &\label{11.9.0.2}
\|\xi[t_0]\|_{\st(t_0,t_1)}\le C(\|\xi\|_{\st(t_0,t_1)}+ \|z\|_{L^{12}(t_0,t_1)}^{3}).
\end{align}

\end{proposition}

\begin{proof}
We first assume \eqref{eq:preprop10}.
By nonadmissible Strichartz (Lemma \ref{lem:2}) with $(4,6)$ and $(8/5,4)$ for the $|\xi|^2\xi$ term and $(4,6)$ and $(4/3,6)$ for the $|z|^2z   G$ term in \eqref{1}   we have
\begin{equation}\label{11.9.2}
\begin{aligned}
      \|\xi\|_{\st(t_0,t_1)}&\lesssim \|\xi[t_0]\|_{\st(t_0,t_1)}+\|\xi\|_{L^{8}L^{4}(t_0,t_1)}^3+\|z^3G\|_{L^4L^{6/5}(t_0,t_1)}\\&
\lesssim \|\xi[t_0]\|_{\st(t_0,t_1)}+\|\xi\|_{L^{\infty}L^{3}(t_0,t_1)}^{3/2}\|\xi\|_{L^{4}L^{6}(t_0,t_1)}^{3/2}+\|z\|_{L^{12}(t_0,t_1)}^3.
\end{aligned}
\end{equation}
By $H^{1/2}\hookrightarrow L^3$ and \eqref{11.9.-0}  we have \begin{equation}
\begin{aligned} &
    \|\xi\|_{L^{\infty}L^{3}(t_0,t_1)}^{\frac 32}\|\xi\|_{\st(t_0,t_1)}^{\frac 32}\lesssim   \left ( \mathcal{N} _{\frac{1}{2}}   ^{3} \|\xi\|_{\st(t_0,t_1)}\right )  ^{\frac{1}{2}} \|\xi\|_{\st(t_0,t_1)} .
\end{aligned}\nonumber
\end{equation}
  By \eqref{eq:preprop10}  we conclude
\begin{align*}
\|\xi\|_{\st(t_0,t_1)}\lesssim \|\xi[t_0]\|_{\st(t_0,t_1)}+ \|z\|_{L^{12}(t_0,t_1)}^3  .
\end{align*}
Interchanging $\xi$ and $\xi[t_0]$  we obtain
\begin{align}\label{11.10}
\|\xi\|_{\st(t_0,t_1)}+ \|z\|_{L^{12}(t_0,t_1)}^3 \sim \|\xi[t_0]\|_{\st(t_0,t_1)}+ \|z\|_{L^{12}(t_0,t_1)}^3 .
\end{align}
Consider the $g$ in \eqref{6}.
Then, again by the nonadmissible Strichartz (Lemma \ref{lem:2}) with $(4,6)$ and $(8/5,4)$ for $|\xi|^2\xi$ and $(4,6)$, $(4/3,6)$ for $\mathcal R_1$, we have
\begin{align*}
\|g\|_{L^4L^{2,-\sigma}(t_0,t_1)}&\lesssim \|g[t_0]\|_{L^4 L^{2,-\sigma}(t_0,t_1)}+\|\xi\|_{L^{8}L^{4}(t_0,t_1)}^3+\|\mathcal R_1\|_{L^4L^{6/5}(t_0,t_1)},
\end{align*}
Recall that $\mathcal R_1$ is given by \eqref{7}.
Then by  $\|\xi\|_{L^{8}L^{4}}^3\lesssim \|\xi\|_{L^\infty L^3}^{3/2}\|\xi\|_{\st}^{3/2}$,  $\|\mathcal R_1\|_{L^4L^{6/5}}\lesssim  \|z\|_{L^\infty}^4 \|\xi\|_{\st}$ and
\begin{align*}
\|g[t_0]\|_{L^4 L^{2,-\sigma}}\lesssim \|\xi[t_0]\|_{\st}+\|Y[t_0]\|_{L^4L^{2,-\sigma}}\lesssim \|\xi[t_0]\|_{\st}+ |z(t_0)|^3,
\end{align*}
where in the last inequality we used Lemma \ref{lem:7}. Then we have
\begin{align}\label{11.11}
\|g\|_{L^4 L^{2,-\sigma}(t_0,t_1)}\lesssim \|\xi[t_0]\|_{\st(t_0,t_1)} +  |z(t_0)|^3+\( \mathcal{N}_{\frac{1}{2}}^{\frac{3}{2} }\|\xi\|_{\st(t_0,t_1)}^{\frac 1 2}+ \|z\|_{L^\infty(t_0,t_1)}^4\) \|\xi\|_{\st(t_0,t_1)}.%\nonumber
\end{align}
We now estimate the  $L^{12}$ norm of $z$.
We multiply  \eqref{10}  by $|z(t)|^6$ obtaining
\begin{align}\label{11.11.1}
\frac 1 4 \frac{d}{dt}|z(t)|^8 =-\Gamma\frac{1}{2}|z|^{12} + \Im\(\frac 1 2 z|z|^8 (G|g) +|z|^8\bar z\overline{(G|g)}\).
\end{align}
Integrating it from $t_0$ to $t(\leq t_1)$, we have
\begin{align*}
|z(t)|^8+ 2\Gamma \|z\|_{L^{12}(t_0,t)}^{12}\lesssim |z(t_0)|^{8}+\|z\|_{L^{12}(t_0,t)}^{9}\|g\|_{L^4L^{2,-\sigma}(t_0,t_1)}.
\end{align*}
Taking $\sup_{t_0<t<t_1}$ by an elementary argument we obtain
\begin{align}\label{11.12}
\|z\|_{L^\infty(t_0,t_1)}^2+ \Gamma^{1/4}\|z\|_{L^{12}(t_0,t_1)}^{3} \lesssim |z(t_0)|^2+\Gamma^{-3/4} \|g\|_{L^4L^{2,-\sigma}(t_0,t_1)}.
\end{align}
From \eqref{11.10}, \eqref{11.11} and \eqref{11.12} we obtain
\begin{align*}
\|\xi\|_{\st(t_0,t_1)}\lesssim& \Gamma^{-3/4}\|\xi[t_0]\|_{\st(t_0,t_1)}+ |z(t_0)|^2+\Gamma^{-3/4}\(\mathcal{N}_{\frac{1}{2}}^{\frac{3}{2} }\|\xi\|_{\st(t_0,t_1)}^{1/2}+ \|z\|_{L^\infty (t_0,t_1)}^4\) \|\xi\|_{\st(t_0,t_1)}.
\end{align*}
By   $\|z\|_{L^\infty (t_0,t_1)}\le \mathcal{N}_0\ll1$ and the assumption \eqref{eq:preprop10} with sufficiently small $\mu_{1/2}$,  we obtain
\begin{align*}
\|\xi\|_{\st(t_0,t_1)}\lesssim& \Gamma^{-3/4}\|\xi[t_0]\|_{\st(t_0,t_1)}+ |z(t_0)|^2.
\end{align*}
Here, the smallness of $\mathcal N_0$ and $\mu_{1/2}$ depends on $\Gamma$.
However, $\Gamma$ is a fixed constant, see \eqref{10.5}, so there is no harm.
Thus
\begin{align*}
\|z\|_{L^\infty(t_0,t_1)}^2+ \|z\|_{L^{12}(t_0,t_1)}^3\lesssim |z(t_0)|^2+ \|\xi[t_0]\|_{\st(t_0,t_1)}.
\end{align*}
Next we  assume \eqref{eq:preprop20}.
If we take $t>t_0$ sufficiently  close to $t_0$  then \eqref{eq:preprop10} is true. Then   from the above argument we obtain
\begin{align*}
\|\xi\|_{\st(t_0,t)}\lesssim \(\|\xi[t_0]\|_{\st(t_0,t_1)}+ |z(t_0)|^2\)\ll \min \{ \cNh^{-3}, 1 \}.
\end{align*}
Thus, by continuity argument we have \eqref{11.9.0} under  assumption \eqref{eq:preprop20}.

Finally we prove \eqref{11.9.0.2}. Under assumption \eqref{eq:preprop10} we know that we have
\eqref{11.10} which, in turn, implies \eqref{11.9.0.2}.
  If instead we start with assumption \eqref{eq:preprop20} then
  \eqref{11.9.0} implies \eqref{eq:preprop10} which, in turn, by the previous sentence implies
   \eqref{11.9.0.2}.
\end{proof}

\begin{remark} The conclusions of  Lemma \ref{lem:8}  continue to hold with Assumption \eqref{eq:preprop10}
replaced by Assumption \eqref{eq:preprop20} since the latter assumption implies the first one by the argument in Proposition \ref{prop:1}.
\end{remark}

\noindent We need estimates on the  solution of \eqref{2} with $z(t)|z(t)|^2G$ replaced by some more general $F (t)$.
\begin{proposition}\label{prop:3} There exist constants $\mu  _{1/2}>0$ and   $C>0$
 s.t.\ for any   $\xi$ satisfying
\begin{align}\label{11.14.2}
\im \dot \xi = -\Delta \xi + |\xi|^2\xi +F ,
\end{align}
where $F\in L^4L^{6/5}(t_0,t_1)\cap L^2W^{1,6/5}(t_0,t_1)$,
with
\begin{equation*}
 \text{either $\|\xi\|_{\st(t_0,t_1)}\le \mu  _{1/2} \min \{ 1, \cNh^{-3} \}$ or $\|\xi[t_0]\|_{\st(t_0,t_1)}+\|F\|_{L^4L^{6/5}(t_0,t_1)}\le \mu  _{1/2}  \min \{ 1, \cNh^{-3} \}$,}
\end{equation*}
where
\begin{align}\label{11.14.-2}
\cNt:=\|\xi (t_0)\|_{ \Ht } +\|F\|_{ L^2W^{\theta,6/5}(t_0,t_1)} ,
\end{align}
we have
\begin{align}\label{11.15}
1/C\le \frac{\|\xi\|_{\st(t_0,t_1)}+\|F\|_{L^4L^{6/5}(t_0,t_1)}}{\|\xi[t_0]\|_{\st (t_0,t_1)}+\|F\|_{L^4L^{6/5}(t_0,t_1)}}\le C
\end{align}
and
\begin{align}\label{11.16}
\|\xi\|_{\nnorm{t_0}{t_1} }\le C (\sqrt{\mu  _{1/2}}\|\xi[t_0]\|_{\st (t_0,t_1)}+\|F\|_{L^4L^{6/5}(t_0,t_1)}).
\end{align}

\end{proposition}

\begin{proof}  Suppose $\|\xi\|_{\st(t_0,t_1)}\ll     \min \{ \mathcal{N}_{\theta}^{-3}, 1\} $. By \eqref{4.1}  we have
\begin{equation}\label{eq:close1}
\|\xi\|_{\stz ^{\theta}(t_0,t_1)}\lesssim \|\xi(t_0)\|_{ H ^{\theta }}+\|F\|_{ L^2W^{\theta,6/5}(t_0,t_1)}+\|\xi\|_{\st(t_0,t_1)}^2\|\xi\|_{\stz ^{\theta}(t_0,t_1)} .
\end{equation}
Then by $\|\xi\|_{\st(t_0,t_1)}^2<1/2$ we obtain
\begin{equation}\label{eq:close11}
\|\xi\|_{\stz ^{\theta}(t_0,t_1)}\lesssim \mathcal{ N}_{\theta}.
\end{equation}
In particular, under the hypothesis $\|\xi\|_{\st(t_0,t_1)}\ll \min \{ 1, \cNh^{-3} \}$ we have \eqref{eq:close11} for $\theta =1/2$.
From  \eqref{11.9.2}  with     $|z|^2z G$ replaced by $F$  we have
\begin{equation}\label{eq:close2}
\begin{aligned}&
\|\xi\|_{\st(t_0,t_1)}
\lesssim \|\xi[t_0]\|_{\st (t_0,t_1)} +\|\xi\|_{L^{\infty}L^{3}(t_0,t_1)}^{3/2}\|\xi\|_{\st(t_0,t_1)} ^{3/2}+\|F\|_{L^4L^{6/5}(t_0,t_1)}.
\end{aligned}
\end{equation}
By   $\|\xi\|_{\st(t_0,t_1)}\ll \cNh^{-3}$ and \eqref{eq:close11} for $\theta =1/2$  we obtain \eqref{11.15}.

\noindent The proof of \eqref{11.15} under the assumption $\|\xi[t_0]\|_{\st(t_0,t_1)}+\|F\|_{L^4L^{6/5}(t_0,t_1)}\ll \min \{ 1, \cNh^{-3}\}$ follows from the previous case by a  continuity argument similar to that in the proof of Proposition \ref{prop:1}.

\noindent Turning to the proof of  \eqref{11.16}, by \eqref{11.9} and the non admissible Strichartz estimate used in \eqref{11.9.2}, we have
\begin{align*}
\|\xi\|_{\nnorm{t_0}{t_1} }&=\sup_{t_0<s<t<t_1}\| \chi_{(-\infty , t]}
\(|\xi|^2\xi + F \)[s]\|_{\st(t_0,t_2)}\\&
\lesssim \(\cNh^{3/2}\|\xi\|_{\st(t_0,t_1)}^{1/2}\)\|\xi\|_{\st(t_0,t_1)}+\|F \|_{L^4L^{6/5}(t_0,t_1)}.
\end{align*}
\end{proof}

In the course of the proof of Proposition \ref{prop:3} we proved also the following lemma.
\begin{lemma}\label{lem:9}  There exist constants $\mu  _{1/2}>0$ and   $C>0$ s.t.\ for
  $\xi$ satisfying  \eqref{11.14.2} with $F\in L^2W^{1,\frac 6 5}(t_0,t_1)$ and, for $\mathcal{N}_{\theta}$
defined by \eqref{11.14.-2}, s.t.\  $\|\xi\|_{\st(t_0,t_1)}\le \mu  _{1/2} \min \{ 1, \cNh^{-3} \}$,
then
\begin{align}
\|\xi\|_{\stzt(t_0,t_1)}&\le C \mathcal{N}_{\theta} \text{ for all $\theta \in [0,1]$}.\label{11.17}
\end{align}
\end{lemma}

\qed

The following is  (5.14)  Lemma 5.1 in \cite{NakanishiJMSJ}.
\begin{lemma}\label{lem:11}
Fix $T>0$ and suppose $u_n  \stackrel{n\to +\infty}{\rightharpoonup}\varphi\text{ weakly in }H^1$.
Then, we have
\begin{align}\label{12}
\|u_n[0]-\varphi[0]\|_{L^\infty (|t|<T;L^4)}\stackrel{n\to +\infty}{\to }0.
\end{align}

\end{lemma}

\section{Nonlinear perturbation}\label{sec:np}

We first recall that in the proof of Theorem \ref{thm:1}  we have shown that if $\|\xi\|_{\stzo(0,\infty)}<\infty$  then $\xi$ scatters forward.
In the following,  for the $\cNt = \|\xi (t_0)\|_{  \Ht }+ |z (t_0) |$ of \eqref{eq:preprop11}
  and under the assumption $\mathcal{N}_0\ll1$,  we show that $\|\xi\|_{\st(t_0,\infty)}<\infty$ is a sufficient  condition for forward scattering.

\begin{lemma}\label{lem:12}
Let $( \xi ,z)$ be the solution of \eqref{1}--\eqref{2} with $(\xi(t_0),z(t_0))\in H^1\times \C$.
Then, we have
\begin{align*}
( \xi ,z)\text{ scatters forward }\ \Longleftrightarrow \ \|\xi\|_{\mathfrak{st}(t_0,\infty)}<\infty.
\end{align*}
A similar statement holds for backward scattering.
\end{lemma}

\begin{proof}
We first prove $\Leftarrow$.
For $T\gg 1 $  we have $\|\xi\|_{\st(T,\infty)}  \stackrel{T \to +\infty}{\rightarrow}0 $.
Notice that by Lemma \ref{lem:-8} we have $\|\xi\|_{L^\infty H^1(t_0,\infty)}\lesssim \cN_1 ^2$.
Thus for $T $ sufficiently large  we can apply  Lemma \ref{lem:8} and conclude that we have $\|\xi\|_{\stz^1(T,\infty)} <\infty$.
Therefore  by Remark \ref{rem:1} we have scattering forward in time.

\noindent
We next show $\Rightarrow$.
Since $\xi$ scatters, by Strichartz estimates, Lemma \ref{lem:1}, we have
\begin{align}
\|\xi[T]\|_{\st(T,\infty)}&\leq \|\xi[T]-\xi[t_0]\|_{\st(T,\infty)}+\|\xi [t_0]\|_{\st(T,\infty)}\label{13}\\&\lesssim \|\xi(T)-\xi[t_0](T)\|_{H^1}+\|\xi[t_0]\|_{\st(T,\infty)}\to 0\ (T\to \infty).\nonumber
\end{align}
Thus for  $T\gg 1 $  we have
$\(\|\xi[T]\|_{\st(T,\infty)} + |z(T)|^2\)\max \{ 1 , \| \xi(T) \| _{H ^{1/2}}  ^{ 3} \}\ll1 $,
where $\| \xi(T) \| _{H ^{1/2}}$ is uniformly bounded in $T$ by Lemma \ref{lem:-8}.
Then  $\|\xi\|_{\st(T,\infty)}<\infty$ by Proposition \ref{prop:1}.
\end{proof}

\begin{lemma}\label{lem:13}
Let $( \xi ,z)$ be a  solution of \eqref{1}--\eqref{2} with $(\xi(t_0),z(t_0))\in H^1\times \C$. Then there exists a $\mu _0 >0$ s.t.\
  if $\mathcal N_0\le \mu _0$
and if   $(\xi,z)$ scatters forward then
\begin{align*}
\|\xi\|_{\nnorm{T}{\infty}}
+\|\xi\|_{L^2 W^{1,6}(T,\infty)}%+\|\xi[T]\|_{L^2 H^1_6(T,\infty)}
\to 0   \text{ as }T\to +\infty.
\end{align*}
\end{lemma}

\begin{proof}
By Lemma \ref{lem:12} combined with Lemma \ref{lem:8}, it is easy to conclude  that $\|\xi\|_{L^2 W^{1,6}(T,\infty)}\stackrel{T\to +\infty}{\rightarrow}0$.
 We have
\begin{align*}
\|\xi\|_{\nnorm{T}{\infty}}&=\sup_{T<S<T_1<\infty} \|\xi[T_1]_>-\xi[S]\|_{\st(T,\infty)}\\&\leq \sup_{T<S<T_1<\infty}\(\|\xi-\xi[S]\|_{\st(T,T_1)}+\|\xi[T_1]-\xi[S]\|_{\st(T_1,\infty)}\)\\&
\leq \|\xi\|_{\st(T,\infty)}+2\|\xi[S]\|_{\st(T,\infty)}+\sup_{T<T_1}\|\xi[T_1]\|_{\st(T_1,\infty)}.
\end{align*}
Since $(\xi,z)$ scatters, we have $\|\xi\|_{\st(0,\infty)}<\infty$.
Thus, we have $\|\xi\|_{\st(T,\infty)}\stackrel{T\to +\infty}{\rightarrow}0$.  By Proposition \ref{prop:1}, we have $\|z\|_{L^{12}(0,\infty)}<\infty$ and therefore $\|z\|_{L^{12}(T,\infty)}\stackrel{T\to +\infty}{\rightarrow}0$.
Then, by \eqref{11.9.0.2} we have $\|\xi[S]\|_{\st(T,\infty)}\stackrel{T\to +\infty}{\rightarrow}0$ for any $S>T$ uniformly. Hence we conclude
$\|\xi\|_{\nnorm{T}{\infty}}\stackrel{T\to +\infty}{\rightarrow}0$.
\end{proof}

\begin{lemma}\label{lem:scatt1}
Let $( \xi ,z)$ be like in Lemmas \ref{lem:12}--\ref{lem:13} satisfying also the conclusions therein.
Then
\begin{equation}\label{eq:lem:scatt1}
 \lim _{t_0\to +\infty}\| \xi - \xi [t_0] \| _{\stz ^1 (t_0, \infty )}=0.
\end{equation}
\end{lemma}
\begin{proof}
  Proceeding like for the proof of inequality \eqref{11.0} in Theorem \ref{thm:1}
we obtain for $t_1\to +\infty$
\begin{equation}  \begin{aligned}    \|\xi- \xi [t_0] \|_{\stz ^{1} (t_0,\infty)}\lesssim   |z(t_0)| +   ( \|z\|_{L^\infty(t_0,\infty)}^4 +\|\xi\|_{\st(t_0,\infty)}^2 )\| \xi\|_{\stz ^{1} (t_0,\infty)}.
\end{aligned} \nonumber
\end{equation}
and since $|z(t_0)|+ \|\xi\|_{\st(t_0,\infty)}\stackrel{t_0\to +\infty}{\rightarrow}0$ we obtain \eqref{eq:lem:scatt1}.
\end{proof}

We now prepare the long time perturbation estimate.
The following lemmas \ref{lem:17} and \ref{lem:16} correspond  to lemmas 6.3 and 6.4 of Nakanishi \cite{NakanishiJMSJ}.
Lemmas \ref{lem:17} and \ref{lem:16} are used in Claim 7.7 and 7.6 in the proof of Proposition \ref{prop:17} respectively.

\begin{lemma}\label{lem:17} There exist fixed constants $\mu _0>0$, $\mu  _{\frac{1}{2}}>0$ and $C>0$ s.t.\
for any interval $(t_0,t_1)$ and for any solutions  of
\begin{align*}
\im \dot \xi  _j  = -\Delta \xi_j +|\xi_j|^2\xi_j +F_j  \text{  in  $(t_0,t_1)$}
\end{align*}
with $F_j\in L^4L^{6/5}(t_0,t_1)\cap L^2W^{1, \frac{6}{5}}(t_0,t_1)$ for $j=1,2$,
  for $\cN_0\le \mu _0$, for
\begin{align*}
\max_{j=1,2}\(\|\xi_j[t_0]\|_{\st(t_0,t_1)}+\|F_j\|_{L^4L^{6/5}(t_0,t_1)}\)\leq \tilde \delta,\quad \|(\xi_1-\xi_2)[t_0]\|_{\st(t_0,t_1)}+\|F_1-F_2\|_{L^4L^{24/23}(t_0,t_1)}\leq \delta ,
\end{align*}
where
\begin{align*}
\cNt:=\max_{j=1,2}\(\|\xi_j(t_0 )\|_{  \Ht}+\|F_j\|_{L^2 W^{\theta, \frac{6}{5}}(t_0,t_1)}\) ,
\end{align*}
and finally for  $0<\delta\leq \tilde\delta\le \mu  _{\frac{1}{2}} \min(\cNh^{-3},1)$,     we have
\begin{align*}
\|\xi_1-\xi_2\|_{\nnorm{t_0}{t_1} }\le C \tilde \delta^{8/7}\delta^{1/7}\cN_1^{6/7}.
\end{align*}
\end{lemma}

%\begin{remark}  There is an improvement with respect to \cite{NakanishiJMSJ} because  here we have $\tilde \delta^{8/7}$ instead of $\tilde \delta^{4/7}$. This improvement is due to the absence in \eqref{1} of a   term  quadratic in $\xi$.\end{remark}

\begin{proof}
First by Proposition \ref{prop:3} and Lemma \ref{lem:9}, we have
\begin{align*}
\|\xi_j\|_{\st(t_0,t_1)}\lesssim \tilde \delta,\quad \|\xi_j\|_{\stz^1(t_0,t_1)}\lesssim \cN_1.
\end{align*}
Now, since
\begin{align*}
\im (\dot \xi _1-\dot \xi _2) = -\Delta(\xi_1-\xi_2)+|\xi_1|^2\xi_1-|\xi_2|^2\xi_2+F_1-F_2,
\end{align*}
for $t_0<s<t_1$, we have
\begin{align*}
\xi_1-\xi_2=(\xi_1-\xi_2)[s]+\cD \(|\xi_1|^2\xi_1-|\xi_2|^2\xi_2+F_1-F_2\)[s].
\end{align*}
Therefore, by nonadmissible Strichartz with
\begin{align*}
&p_0=4,q_0=\frac{24}{7},\sigma_0=\frac{2}{4}+3\(\frac{7}{24}-\frac1 2 \)=-\frac 1 8,\ |\sigma_0|<\frac 2 3, \sigma_0-\frac{1}{p_0}=-\frac 1 8 -\frac1 4<0,\\
&p_1=4/3,q_1=24,\sigma_1=\frac{3}{2}+3\(\frac{1}{24}-\frac1 2 \)=\frac 1 8,\ |\sigma_1|<\frac 2 3, \sigma_1-\frac{1}{p_1}=\frac 1 8 -\frac3 4<0,\\
&p_2=4,q_2=\frac{24}{9},\sigma_2=\frac{2}{4}+3\(\frac{9}{24}-\frac1 2 \)=\frac 1 8,\ |\sigma_2|<\frac 2 3, \sigma_2-\frac{1}{p_2}=\frac 1 8 -\frac1 4<0,\\
\end{align*}
we have
\begin{align*}
\|\xi_1-\xi_2-(\xi_1-\xi_2)[s]& \|_{L^4L^{\frac {24} 7}(t_0,t_1)} \lesssim \|(\xi_1^2+\xi_2^2)(\xi_1-\xi_2)\|_{L^{ \frac 4 3}L^{\frac {24}{15}}(t_0,t_1)}+\|F_1-F_2\|_{L^4L^{    \frac {24}{23} }(t_0,t_1)}\\&\lesssim (\|\xi_1\|_{\st (t_0,t_1)}^2+\|\xi_2\|_{\st (t_0,t_1)}^2)\|\xi_1-\xi_2\|_{L^4L^{\frac {24}{7} }(t_0,t_1)}+\|F_1-F_2\|_{L^4L^{\frac {24}{23} }(t_0,t_1)}\\& \lesssim \tilde \delta^2\|\xi_1-\xi_2\|_{L^4L^{\frac {24}{7} }(t_0,t_1)}+\delta .
\end{align*}
Thus, we have
\begin{align*}
\|\xi_1-\xi_2\|_{L^4L^{24/7}(t_0,t_1)}\lesssim \|(\xi_1-\xi_2)[t_0]\|_{L^4L^{24/7}(t_0,t_1)}+\delta.
\end{align*}
Next, by Lemma \ref{lem:4}, for $t_0<s<t<t_1$,
\begin{align*}
(\xi_1-\xi_2)[t]_>-(\xi_1-\xi_2)[s]=\cD \chi _{(-\infty , t]}\(|\xi_1|^2\xi_1-|\xi_2|^2\xi_2+F_1-F_2\)[s].
\end{align*}
Thus,
\begin{align*}
\|(\xi_1-\xi_2)[t]_>-(\xi_1-\xi_2)[s]\|_{L^4L^{24/7}(t_0,\infty)}&\lesssim \tilde \delta^2\|\xi_1-\xi_2\|_{L^4L^{24/7}(t_0,t_1)}+\delta\\&\lesssim \tilde\delta^2\|(\xi_1-\xi_2)[t_0]\|_{L^4L^{24/7}(t_0,t_1)}+\delta.
\end{align*}
Finally, since $\|f\|_{\st}\lesssim \|f\|_{L^4L^{24/7}}^{4/7}\|f\|_{\stz^1}^{3/7}$, $\|f\|_{L^4L^{24/7}}\lesssim \|f\|_{\st}^{1/4}\|f\|_{\stz^1}^{3/4}$,
we have
\begin{align*}
\|(\xi_1-\xi_2)[t]_>-(\xi_1-\xi_2)[s]\|_{\st(t_0,\infty)}&\lesssim \|(\xi_1-\xi_2)[t]_>-(\xi_1-\xi_2)[s]\|_{L^4L^{24/7}(t_0,\infty)}^{4/7}\cN_1^{3/7}\\&\lesssim
(\tilde\delta^2\|(\xi_1-\xi_2)[t_0]\|_{L^4L^{24/7}(t_0,t_1)}+\delta)^{4/7}\cN_1^{3/7}\\&\lesssim
(\tilde \delta^2\|(\xi_1-\xi_2)[t_0]\|_{\st(t_0,t_1)}^{1/4}\cN_1^{3/4}+\delta)^{4/7}\cN_1^{3/7}\\&\lesssim
\tilde \delta^{8/7}\delta^{1/7}\cN_1^{6/7}.
\end{align*}
Therefore, we have the conclusion.
\end{proof}

\begin{lemma}\label{lem:16} There exists $\mu _0>0$  s.t.\ for solutions   $(\xi _j,z_j)$ of \eqref{1}--\eqref{2} s.t.\ $\cN_0\le \mu _0$,  where
\begin{align*}
\cNt:=\max_{j=1,2}\(\|\xi_j(t_0 )\|_{  \Ht  }+|z_j(t_0 )| \),\quad \cN_2=\max_{j=1,2}\(\|\xi_j\|_{\st(t_0,t_1)}+ \|z_j\|_{L^{12}(t_0,t_1)}^3\) ,
\end{align*}
then  for any  $\varepsilon>0$  there exists  $\delta_*=\delta_*(\cN_0,\cN_1,\cN_2,\varepsilon)>0$ s.t.\
\begin{align*} &
\|(\xi_1-\xi_2)[t_0]\|_{\st(t_0,t_1)}+ \| z_1|z_1|^2-z_2|z_2|^2\|_{L^4(t_0,t_1)}\leq \delta_* \\& \text{ implies } \|\xi_1-\xi_2\|_{\nnorm{t_0}{t_1} }\leq \varepsilon.
\end{align*}
\end{lemma}

\begin{proof}
For $N\gg 1$ determined below, we decompose $(t_0,t_1)$ into subintervals $I_0,I_1,\cdots,I_N$ s.t.\
\begin{align*}
\|\xi_1\|_{\mathfrak{st}(I_j)}+ \|z_1\|_{L^{12}(I_j)}^3\leq 2 N^{-1/4}\mathcal N_2=:\tilde \delta.
\end{align*}
Let $I_j=(S_j,S_{j+1})$ with $S_0=t_0,\ S_{N+1}=t_1$.
Now, if $\tilde \delta \mathcal N_{1/2}^3\ll1$, which is true for  $N\gg 1$   sufficiently large, we can apply Proposition \ref{prop:1} \eqref{11.9.0.2} and obtain,
\begin{align*}
\|\xi_1[S_j]\|_{\mathfrak{st}(I_j)}\lesssim \|\xi_1\|_{\mathfrak{st}(I_j)}+ \|z\|_{L^{12}(I_j)}^3\leq \tilde \delta.
\end{align*}
Suppose we have
\begin{align*}
\|(\xi_1-\xi_2)[S_0]\|_{\mathfrak{st}(S_0,t_1)}+ \|   z_1|z_1|^2-z_2|z_2|^2\|_{L^4(S_0,t_1)}\leq \delta_0\leq \tilde \delta.
\end{align*}
for some $0<\delta_0$.
Then, using $\|z_j\|_{L^{12}}^3=\|z_j^3\|_{L^4}$, we have
\begin{align*}
\|\xi_2[S_0]\|_{\mathfrak{st}(S_0,t_1)}+ \|z_2\|_{L^{12}(S_0,t_1)}^3\lesssim \tilde \delta.
\end{align*}
Thus, we can apply Lemma \ref{lem:17} and  Lemma \ref{lem:-8}  and conclude
\begin{align*}
\|\xi_1-\xi_2\|_{\nnorm{S_0}{S_1}}\leq C  \mathcal N_1^{12/7}\tilde \delta^{8/7}\delta_0^{1/7}.
\end{align*}
Now, set
\begin{align*}
\delta_1:=\delta_0+C\mathcal N_1^{12/7}\tilde \delta^{8/7}\delta_0^{1/7}.
\end{align*}
By  the definition of Nakanishi's seminorm \eqref{11.7} we have
\begin{align*}
\|(\xi_1-\xi_2)[S_1]\|_{\mathfrak{st}(S_1,t_1)}\leq \|(\xi_1-\xi_2)[S_0]\|_{\mathfrak{st}(S_0,t_1)}+\|\xi_1-\xi_2\|_{\nnorm{S_0}{S_1}}.
\end{align*}
Thus, we have
\begin{align*}
\|(\xi_1-\xi_2)[S_1]\|_{\mathfrak{st}(S_1,t_1)}+ \|z_1|z_1|^2-z_2|z_2|^2\|_{L^4(S_1,t_1)}\leq \delta_1.
\end{align*}
If $\delta_1\leq \tilde \delta$, we can repeat the same argument on $I_1$.
Set
\begin{align*}
\delta_{j+1}:=\delta_j+C\mathcal N_1^{12/7}\tilde \delta^{8/7}\delta_j^{1/7},
\end{align*}
inductively.
Now, for given $\mathcal N_1$ and $\mathcal N_2$, we take $N$ large so that
\begin{align*}
\tilde \delta = 2N^{-1/4}\mathcal N_2 \ll \mathcal N_1^{-3}\leq \mathcal N_{1/2}^{-3}.
\end{align*}
Then, if $\delta_j\leq \tilde \delta$, we have
\begin{align*}
\delta_{j+1}\leq \delta_j^{1/7}(\tilde \delta^{6/7}+C\mathcal N_1^{6/7}\tilde \delta^{8/7})\leq \delta_j^{1/7}.
\end{align*}
Thus, if we set $\delta_0$ sufficiently small to satisfy
\begin{align*}
\delta_{N+1}\leq \delta_0^{\frac{1}{7^{N+1}}}\leq \min(\varepsilon,\tilde \delta),
\end{align*}
we have $\delta_{j}\leq \tilde\delta$ $\forall$ $j=0,1,\cdots,N+1$ and by Lemma \ref{lem:3}  we have the following, completing the proof:
\begin{align*}
\|\xi_1-\xi_2\|_{\nnorm{t_0}{t_1} }&\leq \sum_{j=0}^{N} \|\xi_1-\xi_2\|_{\nnorm{S_j}{S_{j+1}}}\\&
\leq \sum_{j=0}^N (\delta_{j+1}-\delta_j)=\delta_{N+1}-\delta_0<\varepsilon.
\end{align*}
\end{proof}

\section{Linear and Nonlinear  Profile  Decompositions}
\label{sec:profdec}

We first recall the following result on  linear profile decompositions, which is a special case of
a more general result in  Lemma 5.3 of \cite{NakanishiJMSJ}. See also \cite{keraani,holmer}.

\begin{proposition} \label{prop:lindecomp}
Let $\{s_n\}_n\subset \R$ and $\{\xi _{0n} \}_n \subset H^1 _{rad}$ with $\sup_n\|\xi _{0n}\|_{H^1}<\infty$.
Then, passing to a subsequence, there exists $J^*\in \N\cup\{\infty\}$ and $\{s^j_n\}_n\subset \R$ for each $0\leq j<J^*$ the following holds.
\begin{enumerate}
\item $s^0_n=s_n$ and $s^j_n-s^k_n\to \infty$ or $s^j-s^k_n\to -\infty$ as $n\to \infty$ for $j\neq k$.
\item For each $j<J^*$, there exists $\varphi^j\in H^1_{rad}$ s.t.\ $\xi _{0n}[s_n](s_n^j)\stackrel{n\to \infty}{\rightharpoonup} \varphi^j$ weakly in $H^1$.
Further, setting $\lambda^j_n=\varphi^j[s_n^j]$, we have $\lambda_n^j(s_n^k)\stackrel{n\to \infty}{\rightharpoonup} 0$ weakly in $H^1$ for $j\neq k$ and $\varphi^j\neq 0$ for $j>0$.
\item If for each finite $J\leq J^*$ we define $\gamma^J_n$ from the equality
\begin{equation}\label{prop:lindecomp1}
    \xi _{0n}[s_n]=\sum_{j=0}^{J-1}\lambda^j_n +\gamma^J_n ,
\end{equation}
then we have $\gamma^J_n(s^j_n)\stackrel{n\to \infty}{\rightharpoonup} 0$ weakly in $H^1$ for $j<J$.
\item For all $\theta\in[0,1]$ we have the Pythagorean formula, for $ \| f\|_{\dot H^\theta}^2:=  \<(-\Delta)^\theta f,f\>$,
\begin{align}\label{pyth}
\sum_{0\leq j<J}\|\lambda_n^j\|_{L^\infty \dot H^\theta (\R )}^2+\|\gamma^J_n\|_{L^\infty \dot H^\theta (\R )}^2= \|\xi _{0n}\|_{\dot H^\theta}^2+o_n(1), \text{ with $ o_n(1)\stackrel{n\to \infty}{\rightarrow} 0 $.}
\end{align}
\item $\| \<(-\Delta)^\theta \lambda^j_n (t), \lambda^k_n (t)\> \| _{L^\infty (\R )} \stackrel{n\to \infty}{\rightarrow} 0$ ($j\neq k $) and $\|   \<(-\Delta)^\theta \lambda^j_n(t),\gamma^J_n(t)\> \| _{L^\infty (\R )} \stackrel{n\to \infty}{\rightarrow} 0$ ($j<J $).
\item
For $0\leq \theta <1$,
\begin{align*}
\lim_{J\to J^*}\limsup_{n\to \infty}\|\gamma^J_n\|_{[L^\infty L^4(\R ),\mathrm{Stz}^1(\R )]_\theta}=0.
\end{align*}
In particular, $\displaystyle \lim_{J\to J^*}\limsup _{n\to +\infty}\|\gamma^J_n\|_{\mathfrak{st}(\R )}=0.
$
\end{enumerate}
\end{proposition}

%{ \begin{remark} We have slightly changed the conclusion of Lemma 5.3 (i) of \cite{NakanishiJMSJ} from $|s^j_n-s^k_n|\to \infty$ to $s^j_n-s^k_n\to \infty$ or $s^j_n-s^k_n\to -\infty$.However, this can be easily verified by taking subsequence in the construction procedure of $s^j_n$.\end{remark} }

We consider now  a sequence of solutions of \eqref{1}--\eqref{2}. More precisely we consider
the following steps \ref{Assumption:A1}--\ref{Assumption:A9}.

\begin{enumerate}[label={\bf(A:\arabic{*})}, ref={\bf(A:\arabic{*})}]
\item\label{Assumption:A1}We  consider   sequences of  solutions $(\xi _n,z_n)\in C^0(\R,  H^1_{rad}\times \C) $ of \eqref{1}--\eqref{2} s.t.\
\begin{equation}\label{eq:indata}
\text{ $\mathcal{N} _{1}< \infty $   and $\mathcal{N} _{0}\ll 1$ for }    \mathcal{N} _{\theta} :=\sup _{n} \| \xi _n (0) \| _{H ^{\theta} } +   | z _n (0)  |.
\end{equation}
Notice that $\mathcal{N} _{\theta}\le C_\theta  \mathcal{N} _{0}^{1-\theta}\mathcal{N} _1^{\theta}$ for fixed constants $C_\theta$.
We can apply Proposition \ref{prop:lindecomp} for     $s^0_n=0$  and  $\xi _{0n}:= \xi _n (0)$.

\item\label{Assumption:A2} We fix $J$ in the decomposition
of  Proposition \ref{prop:lindecomp}, sufficiently large s.t.\  $\limsup_n\|\gamma^J_n\|_{\mathfrak{st}(\R )} \ll \min \{ 1,\mathcal{N}_1^{-3} \}$ and $\limsup_n\|\gamma^J_n\|_{L^\infty L^4(\R )} \ll \mathcal N_0$. We order the
profiles in Proposition \ref{prop:lindecomp} so that there exists $0< L \le J$ s.t.\ for any $0<j< L$  we have $s^{j}_{n}-s^{j-1}_{n}\stackrel{n\to \infty}{\rightarrow}+\infty$ and for $L \le j < J$ we have $s^{j}_{n}\stackrel{n\to \infty}{\rightarrow}-\infty$.

\item\label{Assumption:A3} We introduce a parameter $\tau>1 $ and set $s^j_{\pm ,n} = s^j_{ n}\pm \tau $, but with $s^0_{- ,n} = 0$ and $s^L_{-,n} =\infty$.

\item\label{Assumption:A4} Reducing to subsequences  we can assume that  $z_n(\cdot +s_n^j)\stackrel{n\to \infty}{\to} z^j$ in $\C$ and  $\xi _n(\cdot +s_n^j)\stackrel{n\to \infty} {\rightharpoonup} \xi^j$ weakly in $H^1$   and  uniformly on compact sets.
\item\label{Assumption:A5} We set
$\lambda^j:=\varphi^j[0]$.
\item\label{Assumption:A6} {We set}
$\Lambda^j_n:=\xi^j(\cdot-s^j_n)$ \textit{and}  $z^j_n:=z^j(\cdot-s^j_n)$.

\item\label{Assumption:A9} {For $0\le j < L $ we denote by
$\Gamma^{J,j,\tau}_n$   the function s.t.\ }
\begin{equation}\label{eq:Jjt}
 \begin{aligned}
&\im \dot \Gamma^{J,j,\tau}_n =-\Delta \Gamma^{J,j,\tau}_n  +\chi _{[s ^{j}_{+,n}   ,s ^{j+1}_{-,n} ]}\( |\Gamma^{J,j,\tau}_n |^2\Gamma^{J,j,\tau}_n+ |z_n|^2 z_n G\)  \\& \Gamma^{J,j,\tau}_n (s^j_{+,n})={\gamma^{J,j,\tau}_n(s^j_{+,n})}.
\end{aligned}
\end{equation}
In the case $j=L-1$, we replace $[s^{j}_{+,n},s^{j+1}_{-,n}]$ by $[s^{L-1}_{+,n},\infty)$.

\end{enumerate}

\begin{definition}\label{eq:littleo}Given a sequence $X_n(\tau)$ dependent on a large parameter $\tau \gg 1 $ we write  $X_n(\tau)=o_\tau$
if   $\displaystyle \lim_{\tau\to +\infty}\limsup _{n\to +\infty} X_n(\tau)=0$.

\end{definition}

\noindent In the sequel we  will have various quantities   and  the relation among them will be
\begin{align*}
o_\tau \ll \sup _n\|\gamma^J_n\|_{\mathfrak{st} (\R )}\ll \max \{\mathcal N_{1}^{-3}, \mathcal N_0\}  \ll1.
\end{align*}
Notice that $\mathcal N_1^3\mathcal N_0$ may  be not   small but $\|\gamma^{J}_n\|_{\st(\R)} \mathcal N_1^3 \ll1$.

\begin{lemma}\label{lem:gamma0}
For any $0\leq j<L$ and $T>0$, we have
\begin{align}\label{gamma0}
\|\gamma^{J}_n \|_{\st(|s^j_n-t|<T)}\to 0\text{ as }n\to 0.
\end{align}
\end{lemma}

\begin{proof}
This is (7.18) of Nakanishi \cite{NakanishiJMSJ}.
\end{proof}
% { We have
%\begin{align}\label{eq:uptau}
%\|\upsilon^{J,\tau}_n\|_{\st(\R_+)}=o_\tau,\quad \|\gamma^{J,j,\tau}_n\|_{\st(\R \backslash (s^{j}_{+,n},s^{j+1}_{-,n}))}=o_\tau,\quad\|\gamma^{J,j,\tau}_n\|_{\st(\R)}\leq \|\gamma^{J}_n \|_{\st(s^j_{+,n},s^{j+1}_{-,n})}+o_\tau
%\end{align}
% where the 1st follows from
% $\|\gamma^{J'}_n\|_{\st(\R)}\le  \tau ^{-1}$ and by   $\| \lambda^{l}_n\|_{\st(\R_+)}\stackrel{n\to \infty}{\to} 0$
%   for the $l$'s s.t.\ $s^l_n \stackrel{n\to \infty}{\to} -\infty$ and  the 3rd from the previous two.  We now consider the 2nd. $\gamma^{J,j,\tau}_n$ is defined by the sum in \eqref{eq:defgammajtau}   with $ \gamma^{J,j,k}_n= \varphi ^{J,j,k} [s^{J,j,k}_n] $, where we have relabeled $\lambda^l_n=\varphi^l[s^l_n]$ as in \ref{Assumption:A7}.
%For all $1\leq k\leq n_{j,\tau}$ in \eqref{eq:defgammajtau}   we have $|s^{J,j,k}_n -s^{\ell}_{n}|  \stackrel{n\to \infty}{\to}+ \infty $ for  $\ell =j,j+1 $. This implies that for each $k$ we have $\|\varphi ^{J,j,k} [s^{J,j,k}_n]\|_{\st(\R \backslash (s^{j}_{+,n},s^{j+1}_{-,n}))} \stackrel{n\to \infty}{\to} 0 $, and so in particular $\|\gamma^{J,j,\tau}_n\|_{\st(\R \backslash (s^{j}_{+,n},s^{j+1}_{-,n}))}\stackrel{n\to \infty}{\to} 0 $. }

 From $ \xi _n(0) = \xi _{0n}=  \sum_{j=0}^{J-1}\lambda^j_n(0) +\gamma^J_n (0) $, by $z_n(0) \stackrel{n\to \infty}{\rightarrow} z^0$,  by the conservation of $\mathbb{M}$ and $\mathbb{E}$ for \eqref{1}--\eqref{2} %,  by the fact that $  \| u [t_0] \| _{L^2}^2$ and  $  \| |\Delta | ^{\frac{\theta}{2}}   u [t_0] \| _{L ^{2}}^2$ are constant
 and
by the Pythagorean equality
\eqref{pyth}   we have, for $o(1)\stackrel{n\to \infty}{\rightarrow}0$,
\begin{equation}\label{eq:exp1} \begin{aligned} &   \mathbb{M}(\xi _n,z _n ) = \mathbb{M}(\xi ^0,z^0 )+\sum_{j=1}^{J-1}
2^{-1}\|  \lambda^j_n \| _{L^2}^2 + 2^{-1} \|  \gamma^J_n  \| _{L^2}^2 +o(1),\\&
 \mathbb{E}(\xi _n,z _n ) = \mathbb{E}(\xi ^0, z^0 )+\sum_{j=1}^{J-1}
2^{-1}\|\nabla \lambda^j_n \| _{L^2}^2 + 2^{-1} \|\nabla \gamma^J_n  \| _{L^2}^2 +o(1).
\end{aligned}
\end{equation}
%
% {
%The following lemma is a direct consequence of Proposition \ref{prop:3} and \eqref{eq:uptau}.
%\begin{lemma}\label{lem:14}
%There exists a fixed $C>0$ s.t.
%for any sufficiently large $\tau$ and for   the $\Upsilon^{J,\tau}_n$ of \eqref{eq:upsilon}  we have
%\begin{equation}\label{eq:lem:141}
% \|  \Upsilon^{J,\tau}_n \| _{\stz ^{\theta} (\R )}\le C \mathcal{N} _{\theta},  \quad \quad \|   \Upsilon^{J,\tau}_n \| _{\st (\R_+ )}=o_\tau .
% \end{equation}
%\end{lemma}
%}

The following lemma is     proved, see formulas (7.15) and (7.30), in Sect. 7 \cite{NakanishiJMSJ}.

\begin{lemma}\label{lem:15}  We have for the $\lambda^k_n$'s  of Proposition \ref{prop:lindecomp} and for $0\le j<L$
\begin{align*}
\sum_{k=0}^{j-1}\|\lambda^k_n\|_{\st(s^j_-,\infty)}+\sum_{k=j+1}^{J-1}\|\lambda^k_n\|_{\st(0,s^j_{+,n})}\stackrel{n\to +\infty}{\rightarrow} 0 \text{ and }
\|\lambda^j_n\|_{\st\(\R\setminus (s^j_{-,n} , s^j_{+,n})\)}=o_\tau   .
\end{align*}
\end{lemma}

\section{The main iteration argument}\label{sec:it}

The following  analogue of Lemma 7.1 \cite{NakanishiJMSJ} is the main
property of profile decompositions.
\begin{proposition}\label{prop:17} Let $0< l \le   L$ with the $L$ of \ref{Assumption:A2} and assume
  that the    $(\xi^j, z^j)$ in  \ref{Assumption:A4} scatter forward  for all $j< l$.
  Let $\ell =\min \{ l, L-1 \}$.
Then the following are true:
\begin{enumerate}
\item[$(i)$]   for $0 \le j \leq \ell $  we have
\begin{align}
\|\xi_n[s^j_{-,n}]-\gamma^{J}_n
-\sum_{i=j}^{J-1}\lambda_n^i \|_{\st(s^j_{-,n},\infty)}=o_\tau,\label{$(i)$-$1$}\\
\|\Lambda^j_n[s^j_{-,n}]_>-\lambda^j_n\|_{\stz^1(0,\infty)}=o_\tau ;\label{$(i)$-$2$}
\end{align}

%
%\item[$(i)$-$2$] for $0\leq j\le \ell $  we have
%\begin{equation}\label{$(i)$-$2$}
%    \|\Lambda^j_n[s^j_{-,n}]_>-\lambda^j_n\|_{\stz^1(0,\infty)}=o_\tau ;
%\end{equation}
\item[$(ii)$]  for $0 \leq j\leq \ell $  we have
\begin{equation}\label{$(ii)$-$23$}
    \|\Lambda^j_n\|_{\nnorm{0}{s^j_{-,n}}} =o_\tau ;
\end{equation}

\item[$(iii)$]  for $0 \leq j< \ell $  we have
\begin{equation}\label{$(ii)$-$23b$}
      \|\Lambda^j_n\|_{\nnorm{s^j_{+,n}}{\infty}}=o_\tau ;
\end{equation}

\item[$(iv)$] for $0 \leq j\leq  \ell $  we have
\begin{equation}\label{$(ii)$-$1$}
   \|\xi_n-\Lambda^j_n\|_{\nnorm{s^j_{-,n}}{s^j_{+,n}}}=o_\tau ;
\end{equation}
%
%
%\item[$(ii)$-$3$]  for $0\leq j<\ell $  we have
%\begin{equation}\label{$(ii)$-$3$}
%    \|\Lambda^j_n\|_{\nnorm{s^j_{+,n}}{\infty}}=o_\tau ;
%\end{equation}

\item[$(v)$] for $0\leq j<\ell $  we have
\begin{equation}\label{$(iii)$-$1$}
    \|\xi_n- \Gamma^{J,j,\tau}_n\|_{\nnorm{s^j_{+,n}}{s^{j+1}_{-,n}}}=o_\tau ;
\end{equation}

\item[$(vi)$]
for $0\leq j<\ell$,
\begin{align}\label{22}
|z^{j+1} (-\tau )|+ \|\xi^{j+1}-\varphi^{j+1}[0]\|_{\stz^1(-\infty,-\tau)}\to 0,\quad \text{as }\tau\to \infty.
\end{align}
\item[$(vii)$]
for $0\leq j<\ell$,
\begin{align}\label{GammaAf}
\|\Gamma^{J,j,\tau}_n-\gamma^J_n\|_{\st(s^{j+1}_{-,n},\infty)}=o_\tau.
\end{align}
\end{enumerate}

\end{proposition}

\begin{remark}
We prove Proposition \ref{prop:17} by induction.  First  we prove \eqref{$(i)$-$1$} and \eqref{$(i)$-$2$} for $j=0$, which are trivial, and then we prove $(i)\Rightarrow (ii)\Rightarrow \cdots \Rightarrow (vii)$ $\Rightarrow$ ($(i)$ for $j+1$).
Therefore,  step by step (finite induction), we have the conclusion.
However, for $(vii)$, we  specify that $\xi^{j+1}$ scatters backward to $\varphi^{j+1}$ only after we have \eqref{$(i)$-$1$} of $(i)$ for $j+1$.
\end{remark}

\begin{proof}  The proof of Proposition \ref{prop:17} is the consequence of Claims \ref{claim:1}--\ref{claim:8}.

\begin{claim}\label{claim:1}\eqref{$(i)$-$1$} and \eqref{$(i)$-$2$} are true for $j=0$.
\end{claim}
\proof Claims \eqref{$(i)$-$1$} and \eqref{$(i)$-$2$} for $j=0$  are true because the l.h.s.\ are $0$ by definition. \qed

\begin{claim}[Proof of $(ii)$ for $j$]\label{claim:2}Assume \eqref{$(i)$-$1$} and \eqref{$(i)$-$2$} for a  $j $ with $j \le \ell$. Then   \eqref{$(ii)$-$23$}   is true for   $j $.
\end{claim}
\proof
The claim follows from
\begin{align*}\|\Lambda^j_n \|_{\nnorm{0}{\sjm}}&=\|\Lj-\lj\|_{\nnorm{0}{\sjm}}\lesssim\|\Lj-\lj\|_{\stz^1(0,\sjm)}\leq \| \Lj[\sjm]_>-\lj\|_{\stz^1(0,\infty)}=o_\tau ,
\end{align*}
where we have used \eqref{11.8.1} as well as \eqref{$(i)$-$2$}.\qed

\begin{claim}[Proof of $(iii)$ for $j$]\label{claim:21}Assume \eqref{$(i)$-$1$} and \eqref{$(i)$-$2$} for a  $j $ with $j < \ell$. Then   \eqref{$(ii)$-$23b$}   is true for   $j $.
\end{claim}
\proof
By Lemma \ref{lem:13}  and the hypothesis  that $( \xi ^j, z^j )$ is scattering forward   for $0\le j <\ell$,   by the definition of $\Lj$  in \ref{Assumption:A6} we have
\begin{align*}
\|\Lj\|_{\nnorm{\sjp}{\infty}}=\| \xi^j \|_{\nnorm{\tau}{\infty}} = o_\tau \text{  for $0\le j < \ell$ }.
\end{align*}
\qed

\begin{claim}[Proof of {$(iv)$}  for $j$]\label{claim:4}
Assume \eqref{$(i)$-$1$} and \eqref{$(i)$-$2$} for a $j$ with $j\le \ell$.
Then  \eqref{$(ii)$-$1$} is true for $j$.
\end{claim}
\proof   We have
 \begin{equation} \begin{aligned}\label{xiLambda}    \|(\xi_n-\Lj)[\sjm]\|_{\st(\sjm,\sjp)}\leq &\|\xi_n[\sjm]-\gamma^J_n-\sum_{i=j}^{J-1}\lambda^i_n\|_{\st(\sjm,\sjp)} +\|\gamma^{J}_n\|_{\st(\sjm,\sjp)}\\&+\sum_{i= j+1}^{J-1}\|\lambda^i_n\|_{\st(\sjm,\sjp)}+\|(\Lj-\lj)[\sjm]\|_{\st(\sjm,\sjp)}\\=&o_\tau,\end{aligned}
\end{equation}
where we used  the following bounds for the terms in the r.h.s.: \eqref{$(i)$-$1$} for $j$ for the 1st;    Lemma \ref{lem:gamma0} for the 2nd; Lemma \ref{lem:15} for the 3rd;    \eqref{$(i)$-$2$}  for the  4th.
Therefore by $\||z_n|^2z_n -z^j_n|z^j_n|^2\|_{L^4(\sjm,\sjp)}\stackrel{n\to \infty}{\rightarrow} 0$, which follows from \ref{Assumption:A4} and \ref{Assumption:A6},   we can apply Lemma \ref{lem:16} and obtain \eqref{$(ii)$-$1$}  for $j$.\qed

\begin{claim}[Proof of {$(v)$}    for $j$]\label{claim:5}
Assume \eqref{$(i)$-$1$} and \eqref{$(i)$-$2$} and \eqref{$(ii)$-$1$} for a $j$ with $j<\ell$.
Then \eqref{$(iii)$-$1$} is true  for $j$.
\end{claim}
\proof
Because of  forward scattering of $\xi^j$ and by \eqref{13} for $0\le j<l$ we have
\begin{align}\label{15}
\|\Lj[\sjp]\|_{\st(\sjp,\infty)}=\|\xi^j[\tau ]\|_{\st(\tau,\infty)}\stackrel{\tau \to + \infty}{\rightarrow} 0.
\end{align}
We have  for $0\le j <l$
\begin{align}
& \| \xi_n[\sjp]-\gamma^J_n-\sum_{i=j}^{J-1} \lambda^i_n\|_{\st(\sjp,\infty)} \leq
  \|\xi_n[\sjm]-\gamma^J_n-\sum_{i=j}^{J-1} \lambda^i_n \|_{\st(\sjp,\infty)}\nonumber \\& +\| \lambda^j_n- \Lambda ^{j}_n [\sjm]\|_{\st(\sjp,\infty)}
 +\|(\xi_n-\Lj)[\sjm]-(\xi_n-\Lambda^j_n)[\sjp]\|_{\st(\sjp,\infty)}\label{15.1}
\\& +\|\Lambda^j_n[\sjp]\|_{\st(\sjp,\infty)}+\|\lambda^j_n\|_{\st(\sjp,\infty)}=o_\tau,\nonumber
\end{align}
where we have used  the following bounds for the terms in the r.h.s.: \eqref{$(i)$-$1$} for $j$ for  the 1st and \eqref{$(i)$-$2$} for the 2nd;    \eqref{11.8} and   \eqref{$(ii)$-$1$} for the 3rd; \eqref{15} for the 4th; Lemma \ref{lem:15} for the last.
Therefore  by Lemma \ref{lem:15} and by \eqref{eq:Jjt}
\begin{align}\label{16}
\| \big (\xi_n-{\Gamma^{J,j,\tau }_n}\big )[\sjp]\|_{\st(\sjp, s^{j+1}_{-,n})}=\| \xi_n[\sjp]-\gamma^J_n\|_{\st(\sjp, s^{j+1}_{-,n})}=o_\tau  .
\end{align}
Thus
\begin{equation}  \label{17}  \begin{aligned} & \|\xi_n[\sjp]\|_{\st(\sjp, s^{j+1}_{-,n})}\le    \|\gamma^{J}_n\|_{\st(s^j_{+,n},s^{j+1}_{-,n})} +o_\tau.
\end{aligned}
\end{equation}
By \eqref{17},   forward scattering of $(\xi^j, z^j)$ and    uniform convergence on compact sets  $z_n( \cdot +s_n^j)\to z^j$,   picking
$J\gg1$ and  $\tau\gg 1$  we have
\begin{equation}\label{eq:18-1}
\|\xi_n[\sjp]\|_{\st(\sjp, s^{j+1}_{-,n})}+|z_n(\sjp )|^2\ll \min \{ 1,\cNh^{-3}\} .
\end{equation}
Thus by Proposition \ref{prop:1} and \ref{Assumption:A2}, for $\tau\gg 1$ and $n\gg1$ we have
\begin{align}\label{18}
\|\xi_n\|_{\st(\sjp, s^{j+1}_{-,n})}+\|z_n\|_{L^\infty(\sjp, s^{j+1}_{-,n})}^2+\|z_n\|_{L^{12}(\sjp, s^{j+1}_{-,n})}^3\lesssim \|\gamma_n^J\|_{\st(s^j_{+,n},s^{j+1}_{-,n} )}+o_\tau
\ll \min \{ 1,\cNh^{-3}\}.
\end{align}
%Since we have chosen
%$J\gg1$, by Proposition \ref{prop:lindecomp} and \eqref{eq:indata}  we can assume r.h.s.\eqref{18}$\ll \min \{ 1,\cNh^{-3}\}$  for $n\gg 1$.
Then   \eqref{$(iii)$-$1$} is obtained  from Lemma \ref{lem:17}.

\qed

We  record that from \eqref{18} and Lemma \ref{lem:8}, we have
\begin{align}\label{eq:lem:hypF1}
\|z_n\|_{L^{6 }(\sjp, s^{j+1}_{-,n})}^3 \le
 C_0 \mathcal{N}_0.
\end{align}

\begin{claim}[Partial proof of $(vi)$ for $j$]\label{claim:backscat}
Assume \eqref{18} for a $j$ with $j< \ell$.
%(which follows from \eqref{$(i)$-$1$} and \eqref{$(i)$-$1$} for $j$).
Then, there exists some $h_-^{j+1}\in H^1$ s.t.\ we have
\begin{align}\label{22partial}
|z^{j+1}(-\tau)|+\|\xi^{j+1}-h^{j+1}_-[0]\|_{\stz^1(-\infty,-\tau)}\to 0\text{ as }\tau \to \infty.
\end{align}
\end{claim}

\begin{remark}
To get  \eqref{22}  we need to show $h^{j+1}_- =\varphi^{j+1}$.
This will done after we show \eqref{$(i)$-$1$} for $j+1$.
\end{remark}

\proof
By Lemmas \ref{lem:12} and   \ref{lem:scatt1}, we only have to show $\|\xi^{j+1}\|_{\st(-\infty,0)}<\infty$.
Thus, it suffices to show that for some $\tau>0$, we have
$\|\xi^{j+1}\|_{\st(-T,-\tau)}\leq 1$ for arbitrary $T>\tau$.
Since $\xi_n(s^{j+1}_n+t)\rightharpoonup \xi^{j+1}(t)$, by weak lower semi-continuity and by \eqref{18}, we have
\begin{align*}
\|\xi^{j+1}\|_{\st(-T,-\tau)}\leq \liminf_{n\to \infty} \|\xi_n\|_{\st(s^{j+1}_{n}-T,s^{j+1}_{-,n})}\leq \liminf_{n\to \infty} \|\xi_n\|_{\st(s^{j}_{+,n},s^{j+1}_{-,n})}\leq 1.
\end{align*}
Therefore, we have the conclusion.
\qed

The proof of  \eqref{GammaAf}  follows from   Claims \ref{lem:LinftyL4}--\ref{claim:fin7.7}.
\begin{claim}\label{lem:LinftyL4}
Assume \eqref{18} and \eqref{eq:lem:hypF1} for a $j$ with $j<\ell$.
 %(which follows from \eqref{$(i)$-$1$} and \eqref{$(i)$-$1$} for $j$).
Then  $\|\Gamma^{J,j,\tau}_n \|_{L^\infty (L^2\cap L^4;(s^j_{+,n},\infty))}\lesssim \mathcal N_0$.
\end{claim}

\proof
First, \eqref{18} combined with Lemma \ref{lem:9} yield
\begin{align*}
\|\Gamma^{J,j,\tau}_n \|_{L^\infty  L^2(s^j_{+,n},\infty)}\leq \|\Gamma^{J,j,\tau}_n \|_{\stz^0(s^j_{+,n},\infty)}\lesssim \mathcal N_0.
\end{align*}
We next estimate  $\|\Gamma^{J,j,\tau}_n \|_{L^\infty  L^4(s^j_{+,n},\infty)}$.
By Duhamel's formula
\begin{align}\label{LinftyL4est1}
\Gamma^{J,j,\tau}_n(t) =\gamma^J_n(t) +\mathcal D \(1_{[s^j_{+,n},s^{j+1}_{-,n}]} \(|z_n|^2 z_n G \)\)[s^j_{+,n}](t)+\mathcal D\( 1_{[s^j_{+,n},s^{j+1}_{-,n}]} \(|\Gamma^{J,j,\tau}_n|^2\Gamma^{J,j,\tau}_n\)\)[s^j_{+,n}](t).
\end{align}
For the first term in the r.h.s.\ of \eqref{LinftyL4est1}, we have $\|\gamma_n^J\|_{L^\infty L^4(\R)}\lesssim \mathcal N_0$ by \ref{Assumption:A2}.
The second term can be bounded by Strichartz's estimates and \eqref{eq:lem:hypF1}.
Indeed, by $\stz^1\hookrightarrow L^\infty L^4$ and \eqref{eq:lem:hypF1}, we have
\begin{align*}
\|\mathcal D \(1_{[s^j_{+,n},s^{j+1}_{-,n}]} \(|z_n|^2 z_n G \)\)[s^j_{+,n}]\|_{L^\infty L^4(s^j_{+,n},\infty)}&\lesssim \|z_n|^2 z_n G \|_{L^2 W^{1,\frac{6}{5}}(s^j_{+,n},s^{j+1}_{-,n})}\lesssim \|z_n\|_{L^6(s^j_{+,n},s^{j+1}_{-,n})}^3\\&\lesssim \mathcal N_0.
\end{align*}
We  handle the last term by bootstrap, that is, we
assume   $\|\Gamma^{J,j,\tau}_n\|_{L^\infty L^4(s^j_{+,n},s^j_{+,n}+T)}\leq C\mathcal N_0$ for sufficiently large $C>0$ (but $C\mathcal N_0\ll 1$), and
then we show that we can replace $C$ by $C/2$, achieving  the  desired conclusion by standard arguments.
The estimates to accomplish this follow. We write
\begin{align*}
&\| \mathcal D\( 1_{[s^j_{+,n},s^{j+1}_{-,n}]} \(|\Gamma^{J,j,\tau}_n|^2\Gamma^{J,j,\tau}_n\)\)[s^j_{+,n}] \|_{L^\infty L^4(s^j_{+,n},s^j_{+,n}+T)}\\&\lesssim \sup_{t\in (s^j_{+,n},s^j_{+,n}+T)} \int_{t-1}^t 1_{[s^j_{+,n},s^{j+1}_{-,n}]} (s) |t-s|^{-3/4} \|\Gamma^{J,j,\tau}_n\|_{L^4}^3\,ds
\\&\quad
+\(\int_{s^j_{+,n}}^{t-1}|t-s|^{-3p'/4}\,ds\)^{1/p'}\(\int_{s^j_{+,n}}^{t-1}1_{[s^j_{+,n},s^{j+1}_{-,n}]}(s)\|\Gamma^{J,j,\tau}_n(s)\|_{L^4}^{2p}\,ds\)^{1/p}\|\Gamma^{J,j,\tau}_n\|_{L^\infty L^4}.
\end{align*}
Then, taking $p=4/3$ ($p'=4$), by $\stz^0\hookrightarrow L^{8/3}L^4$, we have
\begin{align*}
&\| \mathcal D\( 1_{[s^j_{+,n},s^{j+1}_{-,n}]} \(|\Gamma^{J,j,\tau}_n|^2\Gamma^{J,j,\tau}_n\)\)[s^j_{+,n}] \|_{L^\infty L^4(s^j_{+,n},s^j_{+,n}+T)}\\&\lesssim \(\|\Gamma^{J,j,\tau}_n\|_{L^\infty L^4(s^j_{+,n},s^j_{+,n}+T)}^2+\|\Gamma^{J,j,\tau}_n\|_{\stz^0(s^j_{+,n},s^{j+1}_{-,n})}^2\)
\|\Gamma^{J,j,\tau}_n\|_{L^\infty L^4(s^j_{+,n},s^j_{+,n}+T)}.
\end{align*}
Therefore, we have the desired estimate.
\qed

\begin{claim}\label{claim:weight}
Assume \eqref{18}, \eqref{eq:lem:hypF1} and \eqref{22partial} for a $j$ with $j< \ell$. %(which follows from \eqref{$(i)$-$1$} and \eqref{$(i)$-$1$} for $j$).
Set
\begin{equation*}
    \|f\|_w:=\sup_{s^j_{+,n}<t<s^{j+1}_{+,n}}w(t)\|f(t)\|_{L^4+L^\infty} \text{ where}
\end{equation*}
 \begin{align*}
w(t)=\begin{cases} 1 & t>s^{j+1}_{-,n},\\
\<t-s^{j+1}_{-,n}\>^{-\delta} & t\leq s^{j+1}_{-,n}
\end{cases}
\end{align*}
for a preassigned $\delta >0$. Then, we have
\begin{align}\label{Gammaw}
\|\Gamma^{J,j,\tau}_n \|_w =o_\tau.
\end{align}
\end{claim}

\proof
The proof is similar to that of Claim \ref{lem:LinftyL4} because
we estimate the three terms in \eqref{LinftyL4est1}.
First, we have $\|\gamma^J_n\|_w\to 0$ as $n\to \infty$.
Indeed, fix $\varepsilon>0$ arbitrary and take $T>0$ so that $\<T\>^{-\delta}<\varepsilon$. Then
\begin{align*}
\|\gamma^J_n\|_w\leq \sup_{s^j_{+,n}\leq t\leq s^{j+1}_{-,n}-T}w(t)\|\gamma^J_n(t)\|_{L^4}+\sup_{s^{j+1}_--T\leq t\leq  s^{j+1}_{+,n}}\|\gamma^J_n(t)\|_{L^4}
\end{align*}
where the 1st term can be bounded by $\<T\>^{-\delta}\|\gamma^J_n\|_{L^\infty L^4}<\varepsilon$ and the 2nd term converges to $0$ as $n\to \infty$ by Lemma \ref{lem:gamma0}.

We next bound the second term of the r.h.s. of \eqref{LinftyL4est1}
\begin{align}\label{west1}
&\|\mathcal D \(1_{[s^j_{+,n},s^{j+1}_{-,n}]} \(|z_n|^2 z_n G \)\)[s^j_{+,n}]\|_w \nonumber\\&\lesssim \sup_{s^j_{+,n}<t<s^{j+1}_{+,n}} w(t)\(\int_{s^j_+}^{t-1}+\int_{t-1}^t\)1_{[s^j_{+,n},s^{j+1}_{-,n}]}(s)
\min(|t-s|^{-3/2},|t-s|^{-3/4})|z_n(s)|^3\,ds .
\end{align}
For $S>1$ yet to be determined,
we divide the time region in three cases $t<s^{j+1}_{-,n}-S$, $s^{j+1}_{-,n}-S<t<s^{j+1}_{-,n}+S$ and $s^{j+1}_{-,n}+S<t<s^{j+1}_{+,n}$. In the 1st case we consider, using the mass invariant \eqref{eq:mass},
\begin{align*}
&\sup_{s^j_{+,n}<t<s^{j+1}_{-,n}-S} w(t)\int_{s^j_+}^{t-1}1_{[s^j_{+,n},s^{j+1}_{-,n}]}(s)|t-s|^{-3/2}|z_n(s)|^3\,ds\lesssim \< S\>^{-\delta}\|z_n\|_{L^\infty(s^j_{+,n},s^{j+1}_{-,n} )}^3\le \< S\>^{-\delta}\mathbb{M} ^3.\end{align*}
In the 2nd case we consider
\begin{align*}&
\sup_{s^{j+1}_{-,n}-S<t<s^{j+1}_{-,n}+S} w(t)\int_{s^j_+}^{t-1}1_{[s^j_{+,n},s^{j+1}_{-,n}]}(s)|t-s|^{-3/2}|z_n(s)|^3\,ds\\&
\leq\sup_{s^ {j+1}_{-,n}-S<t<s^{j+1}_{-,n}+S} w(t)\(\int_{s^j_+}^{s^{j+1}_{-,n}-S}+\int_{s^{j+1}_{-,n}-S}^{t-1}\)1_{[s^j_{+,n},s^{j+1}_{-,n}]}(s)|t-s|^{-3/2}|z_n(s)|^3\,ds\\&\lesssim \<S\>^{-\delta}\|z_n\|_{L^\infty}^3 + \| z_n\|_{L^\infty (s^{j+1}_{-,n}-S,s^{j+1}_{-,n}+S)}^3\le  \<S\>^{-\delta}\mathbb{M} ^3+o_\tau,
\end{align*}
where we have used backward scattering for $j+1$ in the last inequality (that is, it has been
proved in Claim \ref{claim:backscat} that $\| z^{j+1} \|  _{L^\infty (-\tau -S, -\tau +S) }=o_\tau$ for fixed $S$,
and, using \ref{Assumption:A4} and \ref{Assumption:A6}, we get $\|  z_n \|  _{L^\infty (-\tau -S, -\tau +S) }=o_\tau$). Finally, in 3rd case we consider
\begin{align*}
\sup_{s^{j+1}_{-,n}+S<t<s^{j+1}_{+,n}} w(t)\int_{s^j_+}^{t-1}1_{[s^j_{+,n},s^{j+1}_{-,n}]}(s)|t-s|^{-3/2}|z_n(s)|^3\,ds\lesssim S^{-1/2}\|z_n\|_{L^\infty}^3\le S^{-1/2} \mathbb{M} ^3.
\end{align*}
The term with $\int_{t-1}^t$ in \eqref{west1} can be bounded in similarly.
For the case $t<s^{j+1}_{-,n}-S$, we can use the smallness of $w(t)$, for the case $s^{j+1}_{-,n}-S<t<s^{j+1}_{-,n}+S$, we can use the backward scattering and for the case $s^{j+1}_{-,n}+S<t<s^{j+1}_{+,n}$ the integral becomes $0$.

Finally, we estimate the third term of \eqref{LinftyL4est1} by bootstrap:
\begin{align*}
&\| \mathcal D\( 1_{[s^j_{+,n},s^{j+1}_{-,n}]} \(|\Gamma^{J,j,\tau}_n|^2\Gamma^{J,j,\tau}_n\)\)[s^j_{+,n}](t)\|_{L^4+L^\infty} \\&\lesssim \| \Gamma^{J,j,\tau}_n \|_{L^\infty L^2\cap L^4(s^{j}_{+,n},\infty)}^2 \int_{s^j_{+,n}}^t \min(|t-s|^{-3/4},|t-s|^{-3/2})  \|\Gamma^{J,j,\tau}_n(s)\|_{L^2+L^4}\,ds.
\end{align*}
The conclusion follows from   Claim \ref{lem:LinftyL4}  and
\begin{align*}
w(t)\int_{s^j_{+,n}}^t \min(|t-s|^{-3/4},|t-s|^{-3/2})  w(s)^{-1}\,ds\lesssim1  .
\end{align*}

\qed

\begin{claim}[Proof of $(vii)$ for $j$]\label{claim:fin7.7}
We have \eqref{GammaAf}.
\end{claim}

\proof
By \eqref{18} and our choice $J\gg 1$ we can apply Proposition \ref{prop:3} concluding
\begin{equation*}
    \|\Gamma^{J,j,\tau}_n\|_{\st(s^{j }_{+,n},s^{j+1}_{-,n})}\le  \mu  _{1/2} \min \{ 1, \cNh^{-3} \}
\end{equation*}
for the constant $\cNh$ in \eqref{eq:indata}, a constant that, thanks to the Pythagorean formula \eqref{pyth},
serves also as a bound for $\|\gamma^J_n\|_{L^\infty   H ^{\frac{1}{2}} (\R )}$. Then, by Lemma \ref{lem:9}, for the interval $(s^{j }_{+,n},s^{j+1}_{-,n})$ and the standard Strichartz's estimates of Lemma \ref{lem:1} for $(s^{j+1}_{-,n},\infty )$, for a fixed $C$ and the $\mathcal{N}_1$  in  \eqref{eq:indata} we obtain
\begin{align}\label{GammaNext0}
\|\Gamma^{J,j,\tau}_n\|_{\stz ^1(s^{j }_{+,n},\infty)}\le C \mathcal{N}_1.
\end{align}
We next claim
\begin{align}\label{GammaNext}
\|\Gamma^{J,j,\tau}_n\|_{\st(s^{j+1}_{-,n},s^{j+1}_{+,n})}=o_\tau.
\end{align}
Notice that we have
\begin{align}\label{GammaNext4infty}
\|\Gamma^{J,j,\tau}_n\|_{L^\infty  (L^4+L^\infty ) ( s^{j+1}_{-,n},s^{j+1}_{+,n})}=o_\tau,
\end{align}
from \eqref{Gammaw} and the definition of $\|\cdot\|_w$.
By interpolation $\|f\|_{\st}\leq \|f\|_{L^\infty L^\infty}^{1/3}\|f\|_{L^{8/3}L^4}^{2/3}$, $\|f\|_{\st}\leq \|f\|_{L^\infty L^4}^{1/3}\|f\|_{L^{8/3}L^8}^{2/3}$ and $\stz^1 \hookrightarrow L^{8/3}B^1_{4,2}\hookrightarrow L^{8/3}L^4, L^{8/3}L^8$, we have
\begin{align}\label{GammaNextInterp}
\|\Gamma^{J,j,\tau}_n\|_{\st(s^{j+1}_{-,n},s^{j+1}_{+,n})}&\lesssim \|\Gamma^{J,j,\tau}_n\|_{\stz^1(s^{j+1}_{-,n},s^{j+1}_{+,n})}^{2/3}\|\Gamma^{J,j,\tau}_n\|_{L^\infty (L^4+L^\infty )(s^{j+1}_{-,n},s^{j+1}_{+,n})}^{1/3}=o_\tau.
\end{align}
Therefore we have \eqref{GammaNext}.
%To get \eqref{GammaNext} we split $\Gamma^{J,j,\tau}_n = \gamma^J_n + (\Gamma^{J,j,\tau}_n - \gamma^J_n)$. By
%Lemma \ref{lem:gamma0} we know that \eqref{GammaNext} is true  with $\Gamma^{J,j,\tau}_n$ replaced by $\gamma^J_n$. So now we need to replace  $\Gamma^{J,j,\tau}_n$ with $\Gamma^{J,j,\tau}_n - \gamma^J_n$,
%with the latter the sum of the last two terms in \eqref{LinftyL4est1}.
%Inspection of the proof in
%Claim \ref{claim:weight} shows that we have
%\begin{align}\label{GammaNext1}
%\|\Gamma^{J,j,\tau}_n-\gamma^J_n\|_{L^\infty L^\infty(s^{j+1}_{-,n},s^{j+1}_{+,n})}=o_\tau.
%\end{align}
%Then, by $[L^a L ^{\frac{3}{2}a}, L^\infty L^\infty] _{\theta}= \st$ for $a= 4(1-\theta ) $
%and from $\stz ^{1}\hookrightarrow L^a L ^{\frac{3}{2}a}$ for some $a<4$, yields
%\eqref{GammaNext} thanks to \eqref{GammaNext0}.
By  \eqref{GammaNext}   and  Lemma \ref{lem:gamma0},  to get \eqref{GammaAf} it suffices to prove
\begin{align*}
\|\Gamma^{J,j,\tau}_n-\gamma^J_n\|_{\st(s^{j+1}_{+,n},\infty)}=o_\tau.
\end{align*}
This last formula follows from \eqref{LinftyL4est1} combined with Lemma \ref{lem:referee}.
Indeed, by $\stz^1\hookrightarrow L^6L^{18/5}$, we have, by \eqref{eq:lem:hypF1},  \eqref{GammaNext0} and Lemma \ref{lem:referee}
\begin{equation*}  \begin{aligned}   \|\Gamma^{J,j,\tau}_n-\gamma^J_n\|_{\st(s^{j+1}_{+,n},\infty)}&\lesssim \tau^{-1/4}\(\|z_n\|_{L^6(s^j_{+,n},s^{j+1}_{-,n})}^3+\|\Gamma^{J,j,\tau}_n\| ^3_{\stz^1(s^j_{+,n},s^{j+1}_{-,n})}\)
\\& \lesssim \tau^{-1/4} \mathcal{N}_1  (1+ \mathcal{N}_1 ^2)=o_\tau  .
  \end{aligned}
\end{equation*}
\qed

\begin{claim}[Proof of \eqref{$(i)$-$1$}   for $j+1$]\label{claim:6}Assume all the formulas in the statement of Proposition \ref{prop:17} for $j$, with $j<\ell$.
Then \eqref{$(i)$-$1$} is true for $j+1$.
\end{claim}
\proof
We have
 \begin{equation}\label{19} \begin{aligned} &   \| \xi_n[\sjpm]-\gamma^J_n-\sum_{i=j+1}^{J-1}\lambda^l_n \|_{\st(\sjpm,\infty)}\leq
 \| (\xi_n-\Gamma^{J,j,\tau}_n)[\sjpm]-\sum_{i=j}^{J-1}\lambda^l_n \|_{\st(\sjpm,\infty)}+o_\tau\\&
\leq  \| (\xi_n-\Gamma^{J,j,\tau}_n)[s^j_{+,n}]-\sum_{i=j}^{J-1}\lambda^l_n \|_{\st(\sjpm,\infty)}+\|(\xi_n-\Gamma^{J,j,\tau}_n)[s^j_{+,n}]-(\xi_n-\Gamma^{J,j,\tau}_n)[s^j_{-,n}]\|_{\st(s^{j+1}_{-,n},\infty)}\\&\quad +o_\tau=o_\tau.
\end{aligned}
\end{equation}
where in the first inequality we have used Lemma \ref{lem:15} and \eqref{GammaAf} and in the 2nd inequality we have used  \eqref{15.1} and \eqref{$(iii)$-$1$}.
\qed

\begin{claim}[Back scattering:completed]\label{claim:7}Assume all the formulas in the statement of Proposition \ref{prop:17} for $j$, with $j< \ell$ and assume
that   \eqref{$(i)$-$1$} is true for $j+1$.
Then  \eqref{22} is true.
\end{claim}

\proof
First,
\begin{align}
&\|\xi_n[s^{j+1}_{-,n}]-\lambda^{j+1}_n\|_{\st(s^{j+1}_{-,n},s^{j+1}_{+,n})}   \leq \| (\xi_n-\Gamma^{J,j,\tau}_n)[s^{j+1}_{-,n}]-(\xi_n-\Gamma^{J,j,\tau}_n)[s^{j+1}_{-,n}]\|_{\st(s^{j+1}_{-,n},s^{j+1}_{+,n})}\nonumber\\&
\quad+\|\Gamma^{J,j,\tau}_n[s^{j+1}_{-,n}]-\Gamma^{J,j,\tau}_n[s^j_{+,n}]\|_{\st(s^{j+1}_{-,n},s^{j+1}_{+,n})}
+\|(\xi_n-\sum_{i=j+1}^{J-1}\lambda^j_n-\Gamma^{J,j,\tau}_n)[s^j_{+,n}] \|_{\st(s^{j+1}_{-,n},s^{j+1}_{+,n})}\nonumber\\&\quad
+\sum_{i>j+1}^{J-1}\|\lambda^j_n\|_{\st(s^{j+1}_{-,n},s^{j+1}_{+,n})}=o_\tau.\label{20}
\end{align}
Here, for the 1st term we have used \eqref{11.8} and \eqref{$(iii)$-$1$}, for the 2nd term we have used Lemma \ref{lem:gamma0} and \eqref{GammaNext}.
Notice that we have $\Gamma^{J,j,\tau}_n[s^{j+1}_{-,n}]-\Gamma^{J,j,\tau}_n[s^j_{+,n}]=\Gamma^{J,j,\tau}_n-\gamma^J_n$.
For the 3rd term we have used \eqref{15.1} and for the 4th term we used Lemma \ref{lem:15}.

\noindent Since
\begin{align*}
\xi _n(s^{j+1}_n-\tau)-\lambda^{j+1}_n(s^{j+1}_n-\tau)\wto \xi^{j+1}(-\tau)-\lambda^{j+1}(-\tau)
\end{align*}
and
\begin{align*}
\(\xi^{j+1}(-\tau)-\lambda^{j+1}(-\tau)\)[s^{j+1}_n-\tau](t)=(\xi^{j+1}-\lambda^{j+1})[-\tau](t-s^{j+1}_n),
\end{align*}
we have
\begin{equation}\label{eq:backsc1} \begin{aligned}    \|(\xi^{j+1}-\lambda^{j+1})[-\tau]\|_{\st(-\tau,0)}&=
\|\(\xi^{j+1}(-\tau)-\lambda^{j+1}(-\tau)\)[s^{j+1}_n-\tau]\|_{\st(s^{j+1}_n-\tau,s^{j+1}_n)}\\& \leq
\| \(\(\xi^{j+1}(-\tau)-\lambda^{j+1}(-\tau)\)-\(\xi _n-\lambda^{j+1}_n\)\)[s^{j+1}_n-\tau]\|_{\st(s^{j+1}_n-\tau,s^{j+1}_n)}
\\&\quad +\|\(\xi  _n-\lambda^{j+1}_n\)[s^{j+1}_{-,n}]\|_{\st(s^{j+1}_{-,n},s^{j+1}_n)} =o_\tau \end{aligned}
\end{equation}
where we bound the   term in the 2nd line by Lemma \ref{lem:11} and the following term by \eqref{20}, since $\lambda^{j+1}_n [s^{j+1}_{-,n}]= \lambda^{j+1}_n $.
%
%Thus, since $\lambda^{j+1}$ scatters (by definition, since  it is a linear solution), there exists $\tau_*\gg1$ s.t.
%\begin{align*}
%\|\lambda^{j+1}\|_{\st(-\infty,-\tau_*)}+\sup_{\tau>\tau_*}\|\(\xi^{j+1}-\lambda^{j+1}\)[-\tau]\|_{\st(-\tau,0)}\ll \cNh^{-3}.
%\end{align*}
%Hence $\|\xi^{j+1}[-\tau]\|_{\st(-\tau,-\tau_*)}\ll \cNh^{-3}$. We next observe that
% \begin{equation} \begin{aligned} & \|z_n\|_{L^{12}(s^{j+1}_n-\tau, s^{j+1}_{n}-\tau_*)}\le  \|z_n\|_{L^{12}(s^{j}_n+\tau _*, s^{j+1}_{n}-\tau_*)} \lesssim \|\gamma_n^J\|_{\st(\R )}+o _{\tau_*} \ll \cNh^{-3}  \end{aligned}\nonumber
%\end{equation}
%which  follows from \eqref{18} with $\tau$ replaced by $\tau_*$. By
%the locally uniform convergence $ z_n(\cdot +s^{j+1}_n)\stackrel{n\to \infty}{\rightarrow}z^{j+1}$
%  we obtain  $\|z^{j+1}\|_{L^{12}(-\tau, -\tau_*)}\ll \cNh^{-3}$.
%  Since this occurs uniformly for $ \tau > \tau _*$   there exists $\tau_n\to \infty$ s.t.\ $|z^{j+1}(-\tau_n)|^2 \ll \cNh^{-3}$.
%  Then $\|\xi^{j+1}[-\tau]\|_{\st(-\tau,-\tau_*)}+|z^{j+1}(-\tau_n)|^2   \ll \cNh^{-3}$ imply  $\|\xi_n\|_{\st(-\tau_n,-\tau_*)}\ll \cNh^{-3}$ by Proposition \ref{prop:1}, and in turn this implies    we have $\|\xi^{j+1}\|_{\st(-\infty,-\tau_*)}+\|z^{j+1}\|_{L^{12}(-\infty, -\tau_*)}\ll \cNh^{-3}$ and the back scattering of $(z^{j+1}, \xi^{j+1})$ , in particular $z^{j+1}(-\tau )  \stackrel{\tau \to +\infty}{\rightarrow} 0.$  Then there exists $h_- ^{j+1}\in H^1$ s.t.\ by Lemma \ref{lem:1} and the fact that $e  ^{ \im \tau \Delta}$ is an isometry in $H^1$ we have
Now, recall that we have already proved in Claim \ref{claim:backscat} that there exists $h_-^{j+1}\in H^1$ s.t.\
  \begin{equation}\label{eq:backsc12}
  \lim _{\tau \to +\infty }\| \xi^{j+1} [-\tau ]- h_- ^{j+1}[0] \| _{\stz ^1 (-\infty , 0)} \lesssim  \lim _{\tau \to +\infty }\| \xi^{j+1} (-\tau )- e  ^{-\im \tau \Delta} h_- ^{j+1}  \| _{H ^1  } =0.
  \end{equation}
  By \eqref{eq:backsc1} and by $\lambda^{j+1}[-\tau] = \lambda^{j+1}= \varphi ^{j+1} [0]$,
   we have $\displaystyle \lim _{\tau \to +\infty }\| \xi^{j+1}[-\tau]- \varphi ^{j+1} [0]\|_{\st(-\tau,0)} =0$.
Thus we conclude that $h_- ^{j+1}= \varphi ^{j+1} $.
%By an obvious version for backward scattering of Lemma \ref{lem:scatt1}
%   we have $\displaystyle \lim _{\tau \to +\infty }\|  \xi^{j+1}- \xi^{j+1}[-\tau] \|_{\stz ^{1}(-\infty , -\tau) } =0$.
   This completes the proof of    \eqref{22}  for  $j+1$.
\qed

\begin{claim}[Proof of {$(i)$-$2$}  for $j+1$]\label{claim:8}Assume all the formulas in the statement of Proposition \ref{prop:17} for $j$, with    $j<\ell$  and assume
that \eqref{$(i)$-$1$} and \eqref{22} are  true for $j+1$. Then   \eqref{$(i)$-$2$} is true for $j+1$.
\end{claim}
\proof Since by $ \lambda^{j+1}_n := \varphi ^{j+1}[s^{j+1}_n]$  and $ \lambda^{j+1}  := \varphi ^{j+1}[0]$
we have $\lambda^{j+1}_n= \lambda^{j+1}(\cdot-s^{j+1}_n)$,
\begin{align*}
&\|\Lambda^{j+1}_n[s^{j+1}_{-,n}]_> - \lambda^{j+1}_n \|_{\stzo(0,\infty)}=
\|\Lambda^{j+1}_n[s^{j+1}_{-,n}]_> - \lambda^{j+1}(\cdot-s^{j+1}_n) \|_{\stzo(0,\infty)} \\&\le
 \|\xi^{j+1}(\cdot-s^{j+1}_n) - \lambda^{j+1}(\cdot-s^{j+1}_n) \|_{\stzo(0,s^{j+1}_{-,n})}
+\|\Lambda^{j+1}_n[s^{j+1}_{-,n}] - \lambda^{j+1}(\cdot-s^{j+1}_n)\|_{\stzo(s^{j+1}_{-,n},\infty)}
\\&\lesssim
 \|\xi^{j+1} - \lambda^{j+1} \|_{\stzo(-s^{j+1}_{-,n},-\tau)}
+\|\xi^{j+1}(-\tau) - \lambda^{j+1}(-\tau)\|_{H^1}=o_\tau
\end{align*}
where  in the last line we use  \eqref{22}  for  $j+1$  and \eqref{eq:backsc12},  where $e  ^{-\im \tau \Delta} h_- ^{j+1}= e  ^{-\im \tau \Delta} \varphi ^{j+1}=: \lambda^{j+1}(-\tau)$  as shown under \eqref{eq:backsc12}.\qed

The proof of Proposition \ref{prop:17} is completed. \end{proof}

\begin{corollary}\label{lem:boundst}  Assume $l=L$ in Proposition \ref{prop:17}. Then there exists a  fixed constant $C$   s.t.\ \begin{equation} \label{eq:unifbdstdef}\begin{aligned} &
   \|  \xi_n\|_{\st (0,\infty )  }  \le C  . \end{aligned}
\end{equation}
 \end{corollary}

\proof
First by \eqref{18}, we have
\begin{align*}
\sum_{j=0}^{L-1}\|\xi_n\|_{\st(s^j_{+,n},s^{j+1}_{-,n})}\leq \|\gamma^J_n\|_{\st(\R)}+o_\tau\leq 1,
\end{align*}
for $n$ and $\tau$ sufficiently large.
Next, for $0\leq j\leq L-1$, we have
\begin{align*}
\sum_{j=0}^{L-1}\|\xi_n\|_{\st(s^j_{-,n},s^j_{+,n})}\leq \sum_{j=0}^{L-1}\|\xi_n-\Lambda^j_n\|_{\st(s^j_{-,n},s^j_{+,n})}+\|\xi^0\|_{\st(0,\infty)}+\sum_{j=1}^{L-1}\|\xi^j\|_{\st(\R)}.
\end{align*}
The last two terms are bounded so it suffices to bound $\|\xi_n-\Lambda^j_n\|_{\st(s^j_{-,n},s^j_{+,n})}$ for each $j$.
\begin{align*}
\|\xi_n-\Lambda^j_n\|_{\st(s^j_{-,n},s^j_{+,n})}\leq\|\(\xi_n-\Lambda^j_n\)[s^j_{-,n}]\|_{\st(s^j_{-,n},s^j_{+,n})}+\|\xi_n-\Lambda^j_n\|_{[s^j_{-,n},s^j_{+,n}]}=o_\tau.
\end{align*}
Here, we have used \eqref{xiLambda} for the 1st term and \eqref{$(ii)$-$1$} for the 2nd term.
\qed

As in \cite{NakanishiJMSJ} we can formulate the following result, which can be proved similarly.

\begin{proposition}\label{prop:cor} Let
   $(\xi_n,z_n)\in C^0(\R,  H^1_{rad}\times \C ) $  be a sequence of solutions of \eqref{1}--\eqref{2}
satisfying \eqref{eq:indata}. Let
\begin{equation*}
   \xi _{ n}   [0]= \sum_{j=0}^{J-1}\lambda^j_n + \gamma^J_n
\end{equation*}
be the  linearized profile decomposition of Proposition \ref{prop:lindecomp} where $J$ is fixed but large enough. 
Let  $\{s_n^j\}_n$ ($0\leq j<J$) be the sequence given by Proposition \ref{prop:lindecomp} and 
\begin{equation*}
 ( \xi ^j(t),z^j(t)):= \lim _{n\to \infty}(\xi_n,z_n)(t+s^{j}_{n})
\end{equation*}
be the weak limit in $ H^1 _{rad} \times\C$. Assume $(\xi ^j,z^j   )$ scatters as $t\to \sigma \infty$ for each $j<J$ and $\sigma
\in \{ +,- \} $ satisfying $\displaystyle \lim _{n\to \infty }\sigma s^j_n\ge 0$. Then
$\displaystyle \sup _n\| \xi  _n \| _{\st (\R)}<\infty $.

\end{proposition}

\section{Scattering}\label{sec:scattering}
For each $\mu >0$ and $A\in \R$ we denote by $\text{GS}(\mu , A)$ the subset of $C_b ^0(\R , \C \times H^1)$ formed
by the solutions with $ \mathbb{M}\le \mu$ and $\mathbb{E}\le A$.  Let
\begin{equation} \label{induct1}\begin{aligned} & ST(\mu, A) = \sup \{ \| \xi \| _{\st (\R _+)} <\infty : (z,\xi )\in \text{GS}(\mu , A)  \} \\& \mathcal{X} =\{  (\mu , A):  ST(\mu, A)<\infty  \}  .\end{aligned}\nonumber
\end{equation}
We introduce the partial orders  in $\R ^2$
\begin{equation} \label{induct2}\begin{aligned} &(\mu _1, A_1) \le (\mu _2, A_2) \Leftrightarrow     \mu _1 \le \mu _2  \text{ and } A _1 \le A _2 \\&  (\mu _1, A_1) \ll (\mu _2, A_2) \Leftrightarrow     \mu _1 < \mu _2  \text{ and } A _1 < A _2.\end{aligned}\nonumber
\end{equation}
By the definition of $\mathcal{X}$
\begin{equation*}
(\mu _1, A_1) \le (\mu _2, A_2)  \text{ and } (\mu _2, A_2)\in \mathcal{X} \Rightarrow (\mu _1, A_1)\in \mathcal{X}.
\end{equation*}
Our goal is to prove that there exists $\mu _0>0$ s.t.\ $ (0, \mu _0) \times \R \subseteq \mathcal{X} $. By Theorem
\ref{thm:1} we know that there exists   $\delta _0>0$  s.t.\   $ (0, \delta _0) \times  (-\infty , \delta _0) \subseteq \mathcal{X} $.
Suppose there exists $(\mu _0, A_0)\in \R ^2 \backslash \mathcal{X}$  with $\mu _0\ll 1$ and write
\begin{equation} \label{induct3}\begin{aligned} &  E_*= \sup \{ A<A_0: (\mu _0, A)\in \mathcal{X}  \}  ,  \quad M_*= \sup \{ \mu<\mu_0: (\mu _0, E_*)\in \mathcal{X}  \}  .\end{aligned}\nonumber
\end{equation}
Then by Theorem \ref{thm:1}
\begin{equation} \label{induct4}\begin{aligned} &  0<E_*\le A_0   ,  \quad  0< M_*\le \mu_0 ,\end{aligned}\nonumber
\end{equation}
and $(M_*,E_*)$ is s.t.\
\begin{equation} \label{induct5}\begin{aligned} &    (\mu _1, A_1) \lvertneqq  (M_*,E_*) \ll  (\mu _2, A_2)\Rightarrow    (\mu _1, A_1) \in \mathcal{X}  \text{ and }  (\mu _2, A_2) \not\in \mathcal{X} . \end{aligned}
\end{equation}
Hence there is a sequence $ (M_n,E_n)\stackrel{n\to \infty}{\rightarrow} (M_*,E_*)$ and a sequence of solutions $( \xi _n,z_n)\in \text{GS}(M_n,E_n)$ s.t.\
\begin{equation} \label{induct6}\begin{aligned} &    M_n\le \mu _0 +o(1)  \text{  and }  \| \xi _n \|  _{\st (\R _+)} =+\infty  \text{  for all $n$} . \end{aligned}\nonumber
\end{equation}
We can apply to  the sequence  $(\xi _n,z_n)$ the profile decompositions of Section \ref{sec:profdec}. By weak convergence we have
\begin{equation} \label{induct7}\begin{aligned} &    \mathbb{M}( \xi ^j,z^j)\le  M_*  \text{  and }   \mathbb{E}(\xi ^j,z^j)\le  E_*. \end{aligned}
\end{equation}
Since $\| \xi _n \|  _{\st (\R _+)} =+\infty $ for all $n$, by Corollary \ref{lem:boundst}, the assumptions of Proposition \ref{prop:17} must fail and this means that there must exist $l<L$ s.t.\  we have $ \| \xi ^l \|  _{\st (\R _+)} =+\infty $. We choose $l$ minimal, in the sense
that if $\| \xi ^j\|  _{\st (\R _+)} =+\infty $  then $j\geq l$.
%$ s^{j}_n -s^{l}_n  \stackrel{n\to \infty}{\rightarrow} +\infty$.
 By   \eqref{induct5} and \eqref{induct7} we have
\begin{equation} \label{induct8}\begin{aligned} &
 (M_*,E_*)= (\mathbb{M}( \xi ^l,z^l), \mathbb{E}(\xi ^l,z^l)). \end{aligned}\nonumber
\end{equation}
Then $ \xi _n (\cdot + s^{l}_n) \stackrel{n\to \infty}{\rightarrow} \xi ^l$ strongly in $H^1$.
  If $l>0$, \eqref{22} implies $z^l(-s^l_n)\stackrel{n\to \infty}{\rightarrow}0$ in $\C$  and $\xi ^l(-s^l_n)-e^{-\im s^l_n \Delta}\varphi ^{l}\stackrel{n\to \infty}{\rightarrow}0$ in $H^1$.
Since $\lambda ^l_n(0):=e^{-\im s^l_n \Delta}\varphi ^{l} $
  and $\lambda ^l_n(0)\stackrel{n\to \infty}{\rightarrow}0$ in $L^4$
 we get
\begin{equation} \label{induct9}\begin{aligned} &
%M_*= \mathbb{M}( \xi ^l(-s^l_n),z^l(-s^l_n))=  2^{-1}\|  \lambda ^l_n(0) \| ^2 _{L^2}+o(1)\\&
E_*= \mathbb{E}(\xi ^l(-s^l_n),z^l(-s^l_n))= 2^{-1}\| \nabla \lambda ^l_n(0) \| ^2 _{L^2}    +o(1), \end{aligned}\nonumber
\end{equation}
from which we read
\begin{equation} \label{induct10}\begin{aligned} &
2^{-1}\| \nabla \lambda ^l_n(0) \| ^2 _{L^2} \ge  E_*  +  o(1) . \end{aligned}
\end{equation}
Let $( \xi ,z )\in \text{GS}(M_*,E_*)$ with $\| \xi \|  _{\st (\R _+)} =+\infty$.

\begin{claim}\label{claim:relcomp}The
image $(\xi (\R _+),z(\R _+))$ is relatively compact in
$ H^1 _{rad} \times \C$.
\end{claim}
\proof  We consider a sequence $0<t_n \stackrel{n\to \infty}{\rightarrow} +\infty$ and
 we apply the above argument based on Proposition \ref{prop:cor} to $(\xi _n ,z_n) :=(  \xi (\cdot +t_n),z(\cdot +t_n)  ) $    on $(-t_n, 0]$ and on $[0, \infty )$.
Notice that we have $\|\xi_n\|_{\st(-t_n,0)}\to \infty$ as $n\to \infty$ and $\|\xi_n\|_{\st(0,\infty)}=\infty$.
   If in one of the two cases, we have $l=0$ then $\xi _n(0)= \xi ( t_n) $ is strongly convergent  in $H^1$.
   If in both cases $l=l_0>0$ on $(-t_n, 0]$ and  $l=l_1>0$ on $[0, \infty )$ then $(\xi ^0,z^0) $ scatters
   and thus $\mathbb{E}(\xi ^0,z^0) \geq 0$ because if the energy is negative, it cannot scatter. Then using \eqref{eq:exp1} and \eqref{induct10} we have
\begin{equation} \begin{aligned} &    E_*\ge    \mathbb{E}(\xi ^0,z^0 )+2^{-1}\| \nabla \lambda ^{l_0}_n \| ^2 _{L^2}
 + 2^{-1}\| \nabla \lambda ^{l_1}_n \| ^2 _{L^2}+o(1)\ge 2E_* +o(1)
\end{aligned}\nonumber
\end{equation}
so that $E_*\lesssim o(1)$, and since here $o(1)\stackrel{n\to \infty}{\rightarrow}0$, this implies $E_*=0$, in contradiction with
  Theorem \ref{thm:1}  which implies $E_*>0$. As a consequence, up to a subsequence, $\xi _n(0)= \xi ( t_n) $ is strongly convergent  in $H^1$ for any $t_n  \stackrel{n\to \infty}{\rightarrow} +\infty$. \qed

We now prove the following claim, which completes the proof of Theorem \ref{thm:main}.
\begin{claim}\label{claim:contradiction}There are no $( \xi ,z )\in \text{GS}(M_*,E_*)$ with $\| \xi \|  _{\st (\R _+)} =+\infty$.
\end{claim}
\proof We proceed by contradiction assuming the existence of such a solution. By Claim \ref{claim:relcomp}
we know that $\xi (\R _+)\subset H ^{1}_{rad}$ is relatively compact. On the other hand we know that
\begin{equation}\label{eq:contrad1}
   \| \nabla \xi \| _{L^2}^{2}+\frac{3}{4}\|   \xi \| _{L^4}^{4}\ge 2C \gtrsim 1\gg \mu _0
\end{equation}
 because otherwise by Theorem \ref{thm:1} we can show that $\| \xi \|  _{\st (\R _+)} <+\infty$.

\noindent We now consider the Virial Inequality. We consider a smooth function $f(x)=f(|x|)$ with
\begin{equation*}
 f(r) = \left\{\begin{matrix}
     r \text{ for $r\le 1$}\\
 \frac{3}{2} \text{ for  $r\ge 2$ }.
\end{matrix}\right.
\end{equation*}
Then for $f_R(x):=f(x/R)$  and $f_{jR}(x)= f_j(x/R)$ with
\begin{equation*}
 f_0=1-\partial _rf  \, , \quad f_1=\Delta (\partial _r+1/r)f\, , \quad f_2=-3/2+ (\partial _r+1/r)f
\end{equation*}
we have, see \cite{NakanishiJMSJ},
\begin{equation*}\begin{aligned}
 \partial _t \langle Rf_R\xi , \im \partial _r \xi \rangle  &= \| \nabla \xi \| _{L^2}^{2}+\frac{3}{4}\|   \xi \| _{L^4}^{4}-\int _{\R ^3}\left (  2|\partial _r\xi |^2f_{0R}+ R ^{-2}|\xi |^2f_{1R}-|\xi |^4f_{2R} \right ) dx \\& +\langle |z|^2z G, Rf_R\partial _r\xi \rangle  .
\end{aligned}\end{equation*}
Taking $R\gg 1$  by \eqref{eq:contrad1} we obtain the following, which contradicts  $\langle Rf_R\xi , \im \partial _r \xi \rangle \in L^\infty (\R _+)$:
\begin{equation*}\begin{aligned} &
 \partial _t \langle Rf_R\xi , \im \partial _r \xi \rangle  \ge C >0.
\end{aligned}\end{equation*}
  Since
its denial has led to a contradiction, it follows that Claim  \ref{claim:contradiction} is true. \qed

\section*{Acknowledgments}
 S.C. was partially funded  by   a grant   FRA 2015 from the University of Trieste.
M.M. was supported by the JSPS KAKENHI Grant Numbers JP15K17568, JP17H02851 and JP17H02853.
M.M.   thanks Kenji Nakanishi for many valuable advices.
We are grateful for the anonymous referees for giving us a valuable suggestions to improve and simplify the proof.
Especially, for the simplification of the definition of $\Gamma^{J,j,n}_\tau$.

Department of Mathematics and Geosciences,  University
of Trieste, via Valerio  12/1  Trieste, 34127  Italy. {\it E-mail Address}: {\tt scuccagna@units.it}
\\

Department of Mathematics and Informatics,
Faculty of Science,
Chiba University,
Chiba 263-8522, Japan.
{\it E-mail Address}: {\tt maeda@math.s.chiba-u.ac.jp}

\end{document}